\documentclass[a4paper]{amsart}
\usepackage[ansinew]{inputenc}
\usepackage[T1]{fontenc}
\usepackage{longtable}
\usepackage{xcolor}
\usepackage{mathrsfs}
\usepackage{amssymb,amsmath,amsthm}
\usepackage[matrix, arrow, curve]{xy}
\usepackage{tikz-cd}
\usepackage{stmaryrd}
\usepackage{mathtools}
\usepackage{extarrows}
\usetikzlibrary{quotes,babel, angles}

\newcommand{\op}{\textnormal{op}}
\newcommand{\Hom}{\textnormal{Hom}}
\newcommand{\Ext}{\textnormal{Ext}}
\renewcommand{\mod}{\textnormal{mod}\,}
\renewcommand{\pmod}{\underline{\textnormal{mod}}\,}
\newcommand{\imod}{\overline{\textnormal{mod}}\,}
\newcommand{\Mod}{\textnormal{Mod}\,}
\newcommand{\im}{\textnormal{Im}\,}
\newcommand{\coker}{\textnormal{coker}\,}

\newcommand{\Ab}{\textnormal{Ab}}
\newcommand{\fp}{\textnormal{fp}\,}
\newcommand{\Fp}{\textnormal{fp}}
\newcommand{\Ind}{\textnormal{Zsp}\,}
\newcommand{\Inj}{\textnormal{inj}\,}
\newcommand{\radA}{\textnormal{rad}_A}

\numberwithin{equation}{section} \theoremstyle{plain}
\newtheorem*{thm*}{Theorem}
\newtheorem{thm}{Theorem}
\numberwithin{thm}{section}

\newtheorem{coro}[thm]{Corollary}
\newtheorem*{coro*}{Corollary}
\newtheorem{lem}[thm]{Lemma}
\newtheorem*{lem*}{Lemma}
\newtheorem{prop}[thm]{Proposition}
\newtheorem*{prop*}{Proposition}
\newtheorem{rem}[thm]{Remark}
\newtheorem*{rem*}{Remark}
\newtheorem{exa}[thm]{Example}
\newtheorem*{exa*}{Example}

\newtheorem*{df*}{Definition}

\newtheorem*{ques*}{Question}

\newtheorem*{construction*}{Construction}
\newtheorem*{ack*}{ACKNOWLEDGEMENTS}

\begin{document}

\title{Exact Structures and Purity}
\author{Kevin Schlegel}
\date{}

\address{Kevin Schlegel, University of Stuttgart, Institute of Algebra and Number Theory, Pfaffenwaldring 57, 70569 Stuttgart, Germany}
\email{kevin.schlegel@iaz.uni-stuttgart.de}

\subjclass[2020]{16G10, 16D70, 18E20, 18G25, 18G50}
\keywords{Exact category, exact structure, purity, definable subcategory, Ziegler spectrum, Artin algebra, generic module, fp-idempotent ideal}

\begin{abstract} We relate the theory of purity of a locally finitely presented category with products to the study of exact structures on the full subcategory of finitely presented objects. Properties in the context of purity are translated to properties about exact structures. We specialize to the case of a module category over an Artin algebra and show that generic modules are in one to one correspondence with particular maximal exact structures. 

\end{abstract}
\maketitle \section*{Introduction}
Exact structures, in the sense of Quillen \cite{Quillen}, are the right framework to make use of homological methods not only for abelian but also for additive categories. They are prominent in representation theory and appear in, for example, functional analysis \cite{Frerick}, algebraic K-theory \cite{Quillen} and algebraic geometry \cite[Example 13.9]{Buehler}. In general, an exact structure consists of a collection of kernel-cokernel pairs subject to some closure properties. The collection of all split exact sequences forms the smallest exact structure on an additive category. Rump showed that there also always exists a largest exact structure \cite{Rump}.

In this work we build a bridge between the topic of exact structures and the theory of purity, which goes back to the work of Pr\" ufer for abelian groups \cite{Prufer}. For modules over an associative ring it was developed by Cohn to study coproducts of rings \cite{Cohn}. Purity gives a nice framework to understand the structure of certain possibly large modules. It also appears in the context of logic and model theory of modules \cite{Prest}. A systematic treatment of purity for a locally finitely presented category $\mathcal{A}$ with products is due to Crawley-Boevey \cite{Crawley-Boevey2}. We relate this theory of purity to the study of exact structures on the full subcategory $\fp \mathcal{A}$ of finitely presented objects in $\mathcal{A}$. As it turns out, the amount of information that is possible to extract from this connection is bigger, the bigger the largest exact structure on $\fp \mathcal{A}$. In particular, if $\mathcal{A}$ is abelian, then the theory of purity and the study of exact structures are strongly related.

In the context of purity, the Ziegler spectrum of $\mathcal{A}$, denoted by $\Ind \mathcal{A}$, as well as the notion of a definable subcategory of $\mathcal{A}$ is important. Definable subcategories were introduced by Crawley-Boevey for module categories \cite{Crawley-Boevey} and later by Krause for locally finitely presented categories \cite{Krause3}. The Ziegler spectrum was originally introduced by Ziegler \cite{Ziegler}. It is a topological space whose underlying set consists of the isomorphism classes of indecomposable pure-injective objects in $\mathcal{A}$, which are the injective objects in $\mathcal{A}$ relative to the pure-exact structure. Closed sets in $\Ind A$ are in one to one correspondence with definable subcategories of $\mathcal{A}$ by work of Herzog \cite{Herzog} and Krause \cite{Krause2}. Let us explain how to make sense of these notions from the viewpoint of exact structures.

For an exact structure $\mathcal{E}$ on $\fp \mathcal{A}$ we construct an embedding of $\mathcal{A}$ into a Grothen\-dieck category $\mathbf{P}_\mathcal{E}(\mathcal{A})$. The general idea of studying purity via an embedding into a Grothendieck category is due to Gruson and Jensen in the context of module categories \cite{GruJen}; see also Simson \cite{Simson}.
\\ 
\\
\textbf{Theorem A.} (Theorem \ref{big}) \textit{Let $\mathcal{A}$ be a locally finitely presented category with products and $\mathcal{E}$ an exact structure on $\fp \mathcal{A}$. There exists a fully faithful functor $\textnormal{ev}_\mathcal{E} \colon \mathcal{A} \rightarrow \mathbf{P}_\mathcal{E}(\mathcal{A})$ commuting with filtered colimits and cokernels. Moreover, the essential image of $\textnormal{ev}_\mathcal{E}$ is closed under extensions.}
\\

In particular, the above theorem says that we can realize $\mathcal{A}$, up to equivalence, as an extension-closed subcategory of $\mathbf{P}_\mathcal{E}(\mathcal{A})$. This induces an exact structure $\bar{\mathcal{E}}$ on $\mathcal{A}$, which restricts to the exact structure $\mathcal{E}$ on $\fp \mathcal{A}$. For example, if $\mathcal{E}$ is the split exact structure, then $\bar{\mathcal{E}}$ is the pure-exact structure. Using the embedding in Theorem A we show the existence of enough injectives in the exact category $(\mathcal{A}, \bar{\mathcal{E}})$, see Proposition \ref{enough}. We simply call them $\bar{\mathcal{E}}$-injectives. Further, the fp-$\bar{\mathcal{E}}$-injective objects in $\mathcal{A}$ are of interest. They are those $X\in \mathcal{A}$ that admit no non-trivial conflations $X\rightarrow Y \rightarrow C$ in $\bar{\mathcal{E}}$ with $C\in \fp \mathcal{A}$. The following shows how to obtain closed sets in $\Ind \mathcal{A}$ and definable subcategories of $\mathcal{A}$ from exact structures on $\fp \mathcal{A}$.
\\
\\
\textbf{Theorem B.} (Theorem \ref{sum}) \textit{Let $\mathcal{A}$ be a locally finitely presented category with products and $\mathcal{E}_\top$ be the largest exact structure on $\fp \mathcal{A}$. There exist one to one correspondences between the following.
\begin{itemize}
    \item[(1)] Exact structures $\mathcal{E}$ on $\fp \mathcal{A}$.
    \item[(2)] Definable subcategories $\mathcal{X}$ of $\mathcal{A}$ containing all fp-$\bar{\mathcal{E}}_\top$-injectives.
    \item[(3)] Closed sets $\mathcal{U}$ in $\Ind \mathcal{A}$ containing all indecomposable $\bar{\mathcal{E}}_\top$-injectives. 
\end{itemize}
The assignments are given by $\mathcal{E} \mapsto \mathcal{X}_\mathcal{E}$ and $\mathcal{E}\mapsto \mathcal{U}_\mathcal{E}$, where $\mathcal{X}_\mathcal{E}$ denotes the collection of all fp-$\bar{\mathcal{E}}$-injectives and $\mathcal{U}_\mathcal{E}$ the collection of all indecomposable $\bar{\mathcal{E}}$-injectives in $\mathcal{A}$.}
\\

If additionally $\mathcal{A}$ is abelian, then $\mathcal{A}$ is a Grothendieck category and has enough injectives. In this case we precisely obtain all definable subcategories of $\mathcal{A}$ containing all injectives and all closed sets in $\Ind \mathcal{A}$ containing all indecomposable injectives in the above theorem. In Section 3 we proceed by translating properties in the context of purity to properties about exact structures. For example, an exact structure $\mathcal{E}$ on $\fp \mathcal{A}$ is finitely generated, that is $\mathcal{E}$ is the smallest exact structure containing a specific conflation, if and only if $\Ind \mathcal{A} \setminus \mathcal{U}_\mathcal{E}$ is quasi-compact in $\Ind \mathcal{A}$, see Proposition \ref{quasi}.

We specialize to the case that $\mathcal{A}$ is a module category over an Artin algebra $A$, so $\mathcal{A} = \textnormal{Mod}\,A$ and $\fp \mathcal{A} = \mod A$ is the full subcategory of finite length modules. This case is particularly nice, since there are only finitely many isomorphism classes of indecomposable injective $A$-modules $Q$ and $\{Q\}$ is closed and open in $\Ind \mathcal{A}$. Thus, exact structures on $\mod A$ are in correspondence to closed sets in $\Ind \mathcal{A}$ up to a finite choice of indecomposable injective modules by Theorem B. We show a connection between exact structures on $\mod A$ and generic modules. Recall that an indecomposable module $X\in \Mod A$ is generic if it is of infinite length and $X$ is endofinite, that is $X$ is of finite length over its endomorphism ring. The existence of generic modules is related to the second Brauer-Thrall conjecture \cite{Crawley-Boevey3}.
\\
\\
\textbf{Theorem C.} (Theorem \ref{genericex}) \textit{Let $A$ be an Artin algebra. The assignment
\begin{align*}
M \mapsto \mathcal{E}_{M} = \{(f, \coker f) \mid \ker f = 0\text{ and }\coker \Hom_{A}(f,M) = 0\}
\end{align*}
induces a bijection between
\begin{itemize}
    \item[\rm (1)] isomorphism classes of indecomposable endofinite modules $M\in \Mod A$ that are not injective, and
    \item[\rm (2)] maximal exact structures $\mathcal{E}$ on $\mod A$ such that there exists an almost \mbox{$\mathcal{E}$-exact} sequence. 
\end{itemize}
Moreover, $M$ is generic if and only if $\mathcal{E}_M$ contains every almost split sequence.}
\\

An exact structure $\mathcal{E}$ on $\mod A$ is maximal if there is no bigger exact structure than the abelian structure of $\mod A$ and almost $\mathcal{E}$-exact sequences generalize almost split sequences \cite{Auslander00}, see Section 4. 

We continue by studying exact structures $\mathcal{E}$ on $\mod A$ via its associated ideals of morphisms, which were already considered in  \cite[Section 9.2]{Gabriel}. The $\mathcal{E}$-injectivity ideal $\mathcal{I}_\mathcal{E}$ consists of all $ X \rightarrow M$ such that for every conflation $X\rightarrow Y \rightarrow Z$ in $\mathcal{E}$ there exists $Y\rightarrow M$ making the diagram
\begin{equation*}
    \begin{tikzcd}
        M \\
        X \arrow[r]\arrow[u] & Y \arrow[ul]
    \end{tikzcd}
\end{equation*}
commute. This ideal completely determines the exact structure $\mathcal{E}$ and gives rise to a relative Auslander-Reiten formula \cite[Corollary 9.4]{Gabriel}. We show that $\mathcal{I}_\mathcal{E}$ equals the ideal of all morphisms in $\mod A$ that factor through an $\bar{\mathcal{E}}$-injective module in $\Mod A$, see Lemma \ref{ideal}. Moreover, we prove that the assignment $\mathcal{E} \mapsto \mathcal{I}_\mathcal{E}$ gives a one to one correspondence between exact structures on $\mod A$ and fp-idempotent ideals of $\mod A$ that contain all morphisms factoring through an injective $A$-module, see Corollary \ref{exideal}. Fp-idempotent ideals were introduced by Krause in the context of purity for module categories over Artin algebras \cite{Krause}. This makes the importance of fp-idempotent ideals apparent for the study of exact structures on $\mod A$. In Section 5 we give a surprising characterization of fp-idempotent ideals.
\\
\\
\textbf{Theorem D.} (Theorem \ref{new}) \textit{An ideal $\mathcal{I}$ of $\mod A$ is fp-idempotent if and only if there exists a directed system of ideals $\{\mathcal{J}_i\}_{i\in I}$ such that $\mathcal{I} = \bigcap_{i\in I} \mathcal{J}_i^\omega$ with $\mathcal{J}_i^\omega = \bigcap_{n=1}^\infty \mathcal{J}_i^n$.}

\section{Preliminaries}

For a locally finitely presented category $\mathcal{A}$ with products there is a nice theory of purity. We will relate this theory with the study of exact structures on the full subcategory of finitely presented objects of $\mathcal{A}$.

\subsection{Exact structures.} Let $\mathcal{C}$ be an additive category. A \textit{kernel-cokernel pair}
\begin{align*}
    X \xlongrightarrow{f} Y \xlongrightarrow{g} Z \tag{$\ast$}
\end{align*}
in $\mathcal{C}$ is a pair of morphisms $(f,g)$ such that $f$ is the kernel of $g$ and $g$ is the cokernel of $f$. For a collection $\mathcal{E}$ of kernel-cokernel pairs, a morphism $f$ is an \textit{$\mathcal{E}$-monomorphism} if there exists a morphism $g$ such that $(f,g)$ is in $\mathcal{E}$. Similarly, a morphism $g$ is an \textit{$\mathcal{E}$-epimorphism} if there exists a morphism $f$ such that $(f,g)$ is in $\mathcal{E}$. The sequence $(\ast)$ is called \emph{$\mathcal{E}$-exact} if it belongs to $\mathcal{E}$.

An \textit{exact structure} $\mathcal{E}$ on $\mathcal{C}$ is a collection of kernel-cokernel pairs
fulfilling the following properties \cite{Buehler}.
\begin{itemize}
    \item[\mbox{[E0]}] For all $X$ in $\mathcal{C}$ the identity $1_X$ is an $\mathcal{E}$-monomorphism.
    \item[\mbox{[E0$^\text{op}$]}] For all $X$ in $\mathcal{C}$ the identity $1_X$ is an $\mathcal{E}$-epimorphism.
    \item[\mbox{[E1]}] The collection of all $\mathcal{E}$-monomorphisms is closed under composition.
    \item[\mbox{[E1$^\text{op}$]}] The collection of all $\mathcal{E}$-epimorphisms is closed under composition.
    \item[\mbox{[E2]}] The push-out of an $\mathcal{E}$-monomorphism along an arbitrary morphism exists and is again an $\mathcal{E}$-monomorphism.
    \item[\mbox{[E2$^\text{op}$]}] The pull-back of an $\mathcal{E}$-epimorphism along an arbitrary morphism exists and is again an $\mathcal{E}$-epimorphism.
\end{itemize}
As discussed in \cite{Buehler}, this definition of exact structures is equivalent to the one of Quillen \cite{Quillen}.

For example, the collection of all split exact sequences always forms an exact structure $\mathcal{E}_{\bot}$ on $\mathcal{C}$. This is the smallest exact structure. There also always exists a largest exact structure $\mathcal{E}_{\top}$ on an additive category \cite{Rump}.

\subsection{The extension groups.} Let $\mathcal{E}$ be an exact structure on $\mathcal{C}$. For $X,Z\in \mathcal{C}$ the collection of all equivalence classes of $\mathcal{E}$-exact sequences
\begin{align*}
    X \longrightarrow Y \longrightarrow Z
\end{align*}
defines an abelian group $\Ext^1_\mathcal{E}(Z,X)$ (up to set-theoretic issues) with the usual equivalence relation and the Baer sum. If $\mathcal{C}$ is essentially small, then $\Ext^1_\mathcal{E}(Z,X)$ is a set. Moreover, $\Ext^1_\mathcal{E}(-,-)\colon \mathcal{C}^\op \times \mathcal{C} \longrightarrow \Ab$ is a biadditive functor.

An object $Q\in \mathcal{C}$ is \textit{$\mathcal{E}$-injective} if for all $\mathcal{E}$-monomorphisms $f$ the map $\Hom_\mathcal{C}(f, Q)$ is surjective or equivalently $\Ext_\mathcal{E}^1(X,Q) = 0$ for all $X\in \mathcal{C}$. We say that $\mathcal{C}$ has \textit{enough $\mathcal{E}$-injectives} if for all $X \in \mathcal{C}$ there exists an $\mathcal{E}$-injective object $Q$ and an \mbox{$\mathcal{E}$-monomorphism} $X \rightarrow Q$. An $\mathcal{E}$-monomorphism $f\colon X\rightarrow Q$ is an \emph{$\mathcal{E}$-injective envelope} of $X$ if $Q$ is $\mathcal{E}$-injective and $f$ is \emph{left minimal}, which means that for all $\alpha$ if $\alpha f = f$, then $\alpha$ is an isomorphism. The definition of \emph{$\mathcal{E}$-projective objects} and \emph{enough $\mathcal{E}$-projectives} is dual. Similar as for abelian categories, $\mathcal{E}$-injective and \mbox{$\mathcal{E}$-projective} objects can be used to compute the extension groups between objects. For more details, see for example \cite[Chapter 6]{Frerick}.



\subsection{Purity} Let $\mathcal{A}$ be an additive category with filtered colimits, denoted by $\varinjlim$. An object $X\in \mathcal{A}$ is \emph{finitely presented} if $\Hom_\mathcal{A}(X,-)$ commutes with filtered colimits. The category $\mathcal{A}$ is \emph{locally finitely presented} if the full subcategory of finitely presented objects, denoted by $\fp \mathcal{A}$, is essentially small and $\mathcal{A}$ equals the closure of $\fp \mathcal{A}$ under filtered colimits in $\mathcal{A}$, that is $\mathcal{A} = \varinjlim \fp \mathcal{A}$. If additionally $\fp \mathcal{A}$ is abelian, then $\mathcal{A}$ is \emph{locally coherent}. In this case $\mathcal{A}$ is a Grothendieck category.

We fix a locally finitely presented category $\mathcal{A}$ with products and set $\mathcal{C} = \fp \mathcal{A}$. The category $\mathcal{C}$ is idempotent complete and has weak cokernels. For more details, see \cite[Section 1-2]{Crawley-Boevey2}. A sequence 
\begin{align*}
    X\xlongrightarrow{} Y \xlongrightarrow{} Z
\end{align*}
in $\mathcal{A}$ is \emph{pure-exact} if for all $C\in \mathcal{C}$ applying $\Hom_\mathcal{A}(C,-)$ yields an exact sequence
\begin{align*}
    0 \longrightarrow \Hom_\mathcal{A}(C,X) \longrightarrow \Hom_\mathcal{A}(C,Y) \longrightarrow \Hom_\mathcal{A}(C,Z) \longrightarrow 0.
\end{align*}
The pure-exact sequences form an exact structure on $\mathcal{A}$, denoted by $\mathcal{E}_\textnormal{pure}$. The $\mathcal{E}_\textnormal{pure}$-projective objects are precisely the direct summands of coproducts of finitely presented objects. For an exact structure $\mathcal{E}$ on $\mathcal{A}$ an object $X\in \mathcal{A}$ is \emph{fp-$\mathcal{E}$-injective} if $\Ext^1_{\mathcal{E}}(C,X) = 0$ for all $C\in  \mathcal{C}$. As it turns out, every object in $\mathcal{A}$ is fp-$\mathcal{E}_{\textnormal{pure}}$-injective. Most important are the  $\mathcal{E}_\textnormal{pure}$-injective objects in $\mathcal{A}$; the isomorphism classes of indecomposable $\mathcal{E}_\textnormal{pure}$-injectives form a topological space, the \emph{Ziegler spectrum} of $\mathcal{A}$, denoted by $\Ind \mathcal{A}$. We will later describe the topology on $\Ind \mathcal{A}$. For more details, see \cite[Section 3]{Crawley-Boevey2}.

\subsection{The embedding.} The pure-exact structure on $\mathcal{A}$ can also be described by an embedding of $\mathcal{A}$ into an abelian category, which is constructed as follows. A functor is always meant to be additive and covariant. Let $(\mathcal{C}, \Ab)$ be the abelian category of functors $\mathcal{C} \rightarrow \Ab$. Every $F\in (\mathcal{C}, \Ab)$ can be uniquely extended to a functor $\bar{F}\colon \mathcal{A} \rightarrow \Ab$ that commutes with filtered colimits. A functor $F\in (\mathcal{C}, \Ab)$ is \emph{finitely presented} if $F\cong \coker \Hom_{\mathcal{C}}(f,-)$ for a morphism $f$ in $\mathcal{C}$ and in this case $\bar{F} \cong \coker \Hom_{\mathcal{A}}(f,-)$. The full subcategory $\textnormal{fp}(\mathcal{C}, \Ab)$ of finitely presented functors is abelian, since $\mathcal{C}$ has weak cokernels. 

The \emph{purity category} of $\mathcal{A}$, denoted by $\mathbf{P}(\mathcal{A})$, is the category of left exact functors $\textnormal{fp}(\mathcal{C}, \Ab) \rightarrow \Ab$. The functor
\begin{align*}
\textnormal{fp}(\mathcal{C}, \Ab)^\op \longrightarrow \mathbf{P}(\mathcal{A}), \quad F\mapsto \Hom_{\textnormal{fp}(\mathcal{C}, \Ab)}(F,-)    
\end{align*}
induces an equivalence between $\textnormal{fp}(\mathcal{C}, \Ab)^\op$ and $\fp \mathbf{P}(\mathcal{A})$. Moreover, $\mathbf{P}(\mathcal{A})$ is locally finitely presented and since the finitely presented objects form an abelian category, it follows that $\mathbf{P}(\mathcal{A})$ is {locally coherent}. In particular, $\mathbf{P}(\mathcal{A})$ is a Grothendieck category.

\begin{thm*}\cite[(3.3) Lemma 2-3]{Crawley-Boevey2} \textit{Let $\mathcal{A}$ be a locally finitely presented category with products. There exists a fully faithful additive functor
\begin{align*}
    \textnormal{ev} \colon \mathcal{A} \longrightarrow \mathbf{P}(\mathcal{A}),\quad X\mapsto \bar{X} \quad \textit{with} \quad \bar{X}(F) = \bar{F}(X),
\end{align*}
which commutes with filtered colimits and cokernels. Moreover, \textnormal{ev} induces an equivalence between $\mathcal{A}$ and the fp-injective objects in $\mathbf{P}(\mathcal{A})$.}
\end{thm*}

A sequence $X\rightarrow Y \rightarrow Z$ is pure-exact in $\mathcal{A}$ if and only if $\bar{X}\rightarrow \bar{Y} \rightarrow \bar{Z}$ is exact in $\mathbf{P}(\mathcal{A})$. Further, $\textnormal{ev}$ identifies the $\mathcal{E}_\textnormal{pure}$-injectives in $\mathcal{A}$ with the injectives in $\mathbf{P}(\mathcal{A})$. From this one can deduce the existence of $\mathcal{E}_\textnormal{pure}$-injective envelopes in $\mathcal{A}$ from the existence of injective envelopes in the Grothendieck category $\mathbf{P}(\mathcal{A})$. 

The functor $\textnormal{ev}$ restricted to $\mathcal{C} = \fp \mathcal{A}$ coincides with 
\begin{align*}
    \mathcal{C} \longrightarrow \textnormal{fp}(\mathcal{C}, \Ab)^\op, \quad C\mapsto \Hom_\mathcal{C}(C,-)
\end{align*}
under the identification of $\fp \mathbf{P}(\mathcal{A})$ with $\textnormal{fp}(\mathcal{C}, \Ab)^\op$. This embedding of $\mathcal{C}$ into $\textnormal{fp}(\mathcal{C}, \Ab)^\op$ induces the split exact structure on $\mathcal{C}$, since $X\rightarrow Y \rightarrow Z$ is split exact in $\mathcal{C}$ if and only if 
\begin{align*}
   0 \longrightarrow \Hom_\mathcal{C}(C,X) \longrightarrow \Hom_\mathcal{C}(C,Y) \longrightarrow \Hom_\mathcal{C}(C,Z) \longrightarrow 0
\end{align*}
is exact for all $C \in \mathcal{C}$. The equality $\mathcal{A} = \varinjlim \fp \mathcal{A}$ now implies $\mathcal{E}_\textnormal{pure} = \varinjlim \mathcal{E}_\bot$, that is every pure-exact sequence in $\mathcal{A}$ is a filtered colimit of split exact sequences in $\mathcal{C}$. For more details, see \cite[Section 3]{Crawley-Boevey2}. Later, our goal will be to generalize the embedding of $\mathcal{A}$ into an abelian category such that an arbitrary exact structure $\mathcal{E}$ on $\mathcal{C}$ can appear instead of $\mathcal{E}_\bot$.

\subsection{The topology on the Ziegler spectrum.} The topology on $\Ind \mathcal{A}$ actually comes from a topological space associated to the purity category $\mathbf{P}(\mathcal{A})$. Indeed, for every locally coherent category $\mathcal{B}$ one can associate a topological space, the \emph{spectrum} of $\mathcal{B}$, denoted by $\textnormal{Sp}\, \mathcal{B}$. An element in $\textnormal{Sp}\, \mathcal{B}$ is an isomorphism class of an indecomposable injective object in $\mathcal{B}$. A subset $\mathcal{U}$ of $\textnormal{Sp}\,\mathcal{B}$ is closed if there exists a collection $\mathcal{S}$ of objects in $\fp \mathcal{B}$ such that
\begin{align*}
    \mathcal{U} = \{Q\in \textnormal{Sp}\, \mathcal{B} \mid \Hom_\mathcal{B}(X,Q) = 0 \textnormal{ for all }X\in \mathcal{S}\}.
\end{align*}
In fact, the collection $\mathcal{S}$ can always be chosen to be a \emph{Serre subcategory} of $\fp \mathcal{B}$, which means that $\mathcal{S}$ is closed under extensions, subobjects and quotients. For more details, see \cite{Herzog} or \cite{Krause2}. 

The functor $\textnormal{ev}\colon \mathcal{A} \rightarrow \mathbf{P}(\mathcal{A})$ identifies $\Ind \mathcal{A}$ with $\textnormal{Sp}\, \mathbf{P}(\mathcal{A})$ as sets, which induces a topology on $\Ind \mathcal{A}$. It can be described as follows. A full subcategory $\mathcal{X}$ of $\mathcal{A}$ is \emph{definable} if there exists a collection $\mathcal{S}$ of finitely presented functors $\mathcal{C}\rightarrow \Ab$ such that
\begin{align*}
    \mathcal{X} = \{X \in \mathcal{A} \mid \bar{F}(X) = 0\textnormal{ for all }F\in \mathcal{S}\}.
\end{align*}
Now $\mathcal{U} \subseteq \Ind \mathcal{A}$ is closed if there exists a definable subcategory $\mathcal{X}$ of $\mathcal{A}$ with $\mathcal{U} = \Ind \mathcal{A} \cap \mathcal{X}$. For more details, see \cite[Chapter 6]{Krause3}.

\subsection{The correspondences.} For the theory of purity of $\mathcal{A}$, the one to one correspondences between the following collections are fundamental (for more details, see \cite{Crawley-Boevey2}, \cite{Herzog}, \cite{Krause2}, \cite{Krause3} and for a summary \cite[Section 12.1]{Krause4}).   
\begin{itemize}
    \item[(i)] Serre subcategories $\mathcal{S}$ of $\textnormal{fp}(\mathcal{C}, \Ab)$.
    \item[(ii)] Hereditary torsion classes $\mathcal{T}$ of finite type in $\mathbf{P}(\mathcal{A})$. 
    \item[(iii)] Definable subcategories $\mathcal{X}$ of $\mathcal{A}$.
    \item[(iv)] Closed sets $\mathcal{U}$ in $\Ind \mathcal{A}$.
\end{itemize}
Recall that a torsion class is \emph{hereditary} if it is closed under subobjects and of \emph{finite type} if the corresponding torsion-free class is closed under filtered colimits. There are several ways to pass from one collection to the other. We will describe some of them that will be useful for later purposes.

(i)$\leftrightarrow$(ii): Since $\fp \mathbf{P}(\mathcal{A})$ is equivalent to $\textnormal{fp}(\mathcal{C}, \Ab)^\op$, we can identify Serre subcategories of $\textnormal{fp}(\mathcal{C}, \Ab)$ with Serre subcategories $\mathcal{S}$ of $\fp \mathbf{P}(\mathcal{A})$. The assignments are then given by $\mathcal{S} \mapsto \varinjlim \mathcal{S}$ and $\mathcal{T} \mapsto \fp \mathbf{P}(\mathcal{A}) \cap \mathcal{T}$.

(iii)$\leftrightarrow$(iv): The correspondence is given by $\mathcal{X} \mapsto \Ind \mathcal{A} \cap \mathcal{X}$ and $\mathcal{U} \mapsto \mathcal{X}_\mathcal{U}$, where $\mathcal{X}_\mathcal{U}$ is the smallest definable subcategory containing $\mathcal{U}$.

(i)$\leftrightarrow$(iii): The assignments are given by \begin{align*}
    \mathcal{S} &\mapsto \{X \in \mathcal{A} \mid \bar{F}(X) = 0\textnormal{ for all }F\in \mathcal{S}\},\\
    \mathcal{X} &\mapsto \{F\in \textnormal{fp}(\mathcal{C}, \Ab)\mid \bar{F}(X) = 0\textnormal{ for all }X\in  \mathcal{X}\}.
\end{align*}

(ii)$\rightarrow$(iii),(iv): Given $\mathcal{T}$, we obtain a definable subcategory $\mathcal{X}$ of $\mathcal{A}$ and a closed set $\mathcal{U}$ of $\Ind \mathcal{A}$ by the already stated correspondences. The following gives a nice second description of $\mathcal{X}$ and $\mathcal{U}$. Consider the localisation of $\mathbf{P}(\mathcal{A})$ with respect to the hereditary torsion class $\mathcal{T}$ of finite type. Then the composition of the quotient functor $\mathbf{P}(\mathcal{A}) \rightarrow \mathbf{P}(\mathcal{A})/\mathcal{T}$ with $\textnormal{ev}\colon \mathcal{A} \rightarrow \mathbf{P}(\mathcal{A})$ yields a functor $\mathcal{A} \rightarrow \mathbf{P}(\mathcal{A})/\mathcal{T}$. It identifies $\mathcal{X}$ with the fp-injective objects in $\mathbf{P}(\mathcal{A})/\mathcal{T}$ and the $\mathcal{E}_\textnormal{pure}$-injectives in $\mathcal{X}$ with the injectives in $\mathbf{P}(\mathcal{A})/\mathcal{T}$. Moreover, it induces a homeomorphism between $\mathcal{U}$ with its subspace topology and $\textnormal{Sp}\,\mathbf{P}(\mathcal{A})/\mathcal{T}$.  

\section{From exact structures to purity}

We will describe how to enter the theory of purity from exact structures. Let $\mathcal{C}$ be an essentially small additive category. A functor $F\colon \mathcal{C} \rightarrow \Ab$ is \emph{finitely generated} if it is a quotient of a finitely presented functor; $F$ is \emph{effaceable} if $F\cong \coker \Hom_{\mathcal{C}}(f,-)$ for a kernel-cokernel pair
\begin{align*}
    X \xlongrightarrow{f} Y \xlongrightarrow{g} Z
\end{align*}
in $\mathcal{C}$. We denote by $\textnormal{fg}(\mathcal{C}, \Ab)$ and $\textnormal{eff}(\mathcal{C}, \Ab)$ the category of all finitely generated and effaceable functors respectively. Clearly $\textnormal{eff}(\mathcal{C}, \Ab) \subseteq \textnormal{fp}(\mathcal{C}, \Ab) \subseteq \textnormal{fg}(\mathcal{C}, \Ab)$. Note that $\textnormal{eff}(\mathcal{C}^\textnormal{op}, \Ab)$ consists of all functors isomorphic to $\coker  \Hom_\mathcal{C}(-,g)$ for some kernel-cokernel pair $(f,g)$ in $\mathcal{C}$.

\begin{lem}\label{dual} There exists an equivalence
\begin{align*}
E\colon \textnormal{eff}(\mathcal{C}^\textnormal{op}, \Ab) \longrightarrow \textnormal{eff}(\mathcal{C}, \Ab)^\op
\end{align*}
such that
\begin{align*}
E\, \coker \Hom_\mathcal{C}(-,g) \cong \coker \Hom_\mathcal{C}(f,-)    
\end{align*}
for every kernel-cokernel pair $(f,g)$ in $\mathcal{C}$.
\end{lem}
\begin{proof} The functor $E$ is precisely the one in \cite[Lemma 2.5]{Enomoto}. The isomorphism is clear from the proof of \cite[6.2 (1)]{Iyama}.
\end{proof}

A full subcategory $\mathcal{S}$ of $\textnormal{fg}(\mathcal{C}, \Ab)$ is a \emph{Serre subcategory} if for all short exact sequences $0 \rightarrow F \rightarrow G \rightarrow H \rightarrow 0$ of finitely generated functors, the functor $G$ is contained in $\mathcal{S}$ if and only if $F$ and $H$ are contained in $\mathcal{S}$.

For an exact structure $\mathcal{E}$ on $\mathcal{C}$, we denote by $\mathcal{S}_\mathcal{E}$ the collection of all functors isomorphic to $\coker \Hom_\mathcal{C}(f,-)$, where $f$ is an $\mathcal{E}$-monomorphism. For a full subcategory $\mathcal{S}$ of $\textnormal{eff}(\mathcal{C},Ab)$ we denote by $\mathcal{E}_\mathcal{S}$ the collection of all kernel-cokernel pairs $(f,g)$ with $\coker \Hom_\mathcal{C}(f,-) \in \mathcal{S}$. The following theorem is precisely the dual version of \cite[Theorem 2.7]{Enomoto}.

\begin{thm}\label{eno}{\rm \cite[Theorem 2.7]{Enomoto}} Let $\mathcal{C}$ be an essentially small idempotent complete additive category. There exists a one to one correspondence between 
\begin{itemize}
    \item[\rm (1)] exact structures $\mathcal{E}$ on $\mathcal{C}$, and
    \item[\rm (2)] subcategories $\mathcal{S}$ of $\textnormal{eff}(\mathcal{C}, \Ab)$ such that $\mathcal{S}, E\mathcal{S}$ are Serre subcategories of $\textnormal{fg}(\mathcal{C}, \Ab), \textnormal{fg}(\mathcal{C}^\op, \Ab)$ respectively.
\end{itemize} The assignment is given by $\mathcal{E} \mapsto \mathcal{S}_\mathcal{E}$ and $\mathcal{S} \mapsto \mathcal{E}_\mathcal{S}$.
\end{thm}

\begin{coro}\label{enosim} Let $\mathcal{A}$ be a locally finitely presented category with products, set $\mathcal{C} = \fp \mathcal{A}$ and let $\mathcal{E}_{\top}$ be the largest exact structure on $\mathcal{C}$. There exists a one to one correspondence between
\begin{itemize}
    \item[\rm (1)] exact structures $\mathcal{E}$ on $\mathcal{C}$, and 
    \item[\rm (2)] Serre subcategories $\mathcal{S}$ of $\textnormal{fp}(\mathcal{C}, \Ab)$ with $\mathcal{S} \subseteq \mathcal{S}_{\mathcal{E}_{\top}}$.
\end{itemize}
\end{coro}
\begin{proof} Recall that $\mathcal{C}$ is essentially small and idempotent complete, so we can apply Theorem \ref{eno}. Moreover, $\textnormal{fp}(\mathcal{C}, \Ab)$ is abelian since $\mathcal{C}$ has weak cokernels. Every exact structure $\mathcal{E}$ on $\mathcal{C}$ yields a Serre subcategory $\mathcal{S}_\mathcal{E} \subseteq \textnormal{fg}(\mathcal{C}, \Ab)$. Clearly, $\mathcal{S}_\mathcal{E}$ is a Serre subcategory of $\textnormal{fp}(\mathcal{C}, \Ab)$ with $\mathcal{S}_\mathcal{E} \subseteq \mathcal{S}_{\mathcal{E}_\top}$. To apply the other assignment in Theorem \ref{eno}, we show that for a Serre subcategory $\mathcal{S}$ of $\textnormal{fp}(\mathcal{C}, \Ab)$ with $\mathcal{S} \subseteq \mathcal{S}_{\mathcal{E}_\top}$ automatically $E \mathcal{S}$ is a Serre subcategory of $\textnormal{fg}(\mathcal{C}^\op, \Ab)$.  

Let $0 \to F \to G \to H \to 0$ be a short exact sequence of finitely generated functors $\mathcal{C}^\op \to \Ab$. Since $E\mathcal{S}_{\mathcal{E}_\top} \subseteq \textnormal{fg}(\mathcal{C}^\op, \Ab)$ is a Serre subcategory, it follows that $G \in E\mathcal{S}_{\mathcal{E}_\top} \subseteq \textnormal{eff}(\mathcal{C}^\op, \Ab)$ if and only if $F,H \in E\mathcal{S}_{\mathcal{E}_\top}$. Thus, via the duality $E$, we deduce that $G\in E\mathcal{S} \subseteq E \mathcal{S}_{\mathcal{E}^\top}$ if and only if $F,H \in E\mathcal{S}$ because $\mathcal{S} \subseteq \textnormal{fp}(\mathcal{C}, \Ab)$ is a Serre subcategory.
\end{proof}

From now on, we fix a locally finitely presented category $\mathcal{A}$ and set $\mathcal{C} = \fp \mathcal{A}$. The above corollary gives a bridge between the theory of purity of $\mathcal{A}$ and the study of exact structures on $\mathcal{C}$. For an exact structure $\mathcal{E}$ on $\mathcal{C}$, corresponding to the Serre subcategory $\mathcal{S}_\mathcal{E}$ we obtain a hereditary torsion class $\mathcal{T}_\mathcal{E}$ of finite type in $\mathbf{P}(\mathcal{A})$, a definable subcategory $\mathcal{X}_\mathcal{E}$ of $\mathcal{A}$ and a closed set $\mathcal{U}_\mathcal{E}$ in $\Ind \mathcal{A}$ by the assignments in Section 1.6. 

We define the \emph{purity category of $\mathcal{A}$ relative to $\mathcal{E}$} by the localisation of $\mathbf{P}(\mathcal{A})$ with respect to $\mathcal{T}_\mathcal{E}$ and denote it by  $\mathbf{P}_\mathcal{E}(\mathcal{A})$. Now $\mathbf{P}_\mathcal{E}(\mathcal{A})$ is again locally coherent and the quotient functor $q\colon \mathbf{P}(\mathcal{A}) \rightarrow \mathbf{P}_{\mathcal{E}}(\mathcal{A})$ is exact and commutes with filtered colimits. Further, there exists a commutative diagram of functors
\begin{equation*}
    \begin{tikzcd}
        \mathcal{C} \arrow[r] \arrow[d]& \textnormal{fp} (\mathcal{C}, \Ab)^\op \arrow[r] \arrow[d] & \textnormal{fp} (\mathcal{C}, \Ab)^\op/\mathcal{S}_\mathcal{E} \arrow[d]\\
        \mathcal{A} \arrow[r, "\textnormal{ev}"] & \mathbf{P}(\mathcal{A}) \arrow[r, "q"] & \mathbf{P}_\mathcal{E}(\mathcal{A}),
    \end{tikzcd}
\end{equation*}
where the vertical arrows identify the categories in the first row with the finitely presented objects of the categories in the second row. For the first commutative square, see Section 1.4. For the second commutative square and the described properties, see for example \cite[2.6 - 2.8]{Krause2}. The following lemma describes the properties of the composition of the functors in the first row.

\begin{lem}\label{small} There exists a fully faithful functor
\begin{align*}
    \mathcal{C} \longrightarrow \textnormal{fp}(\mathcal{C}, \Ab)^\textnormal{op} / \mathcal{S}_\mathcal{E}, \quad C \mapsto \Hom_\mathcal{C}(C,-),
\end{align*}
whose essential image is closed under extensions.
\end{lem}

\begin{proof} We show that for all $C\in \mathcal{C}$ the functor $\Hom_\mathcal{C}(C,-)$ is $\mathcal{S}_\mathcal{E}$-closed in $\textnormal{fp}(\mathcal{C}, \Ab)$ or equivalently 
\begin{align*}
\Hom_{\textnormal{fp}(\mathcal{C}, \Ab)}(F, \Hom_\mathcal{C}(C,-)) = 0 = \Ext^1_{\textnormal{fp}(\mathcal{C}, \Ab)}(F, \Hom_\mathcal{C}(C,-))    \tag{$\ast$}
\end{align*}
for all $F\in \mathcal{S}_\mathcal{E}$. We have $F\cong \coker \Hom_{\mathcal{C}}(f,-)$ for an $\mathcal{E}$-exact sequence 
\begin{align*}
    X\xlongrightarrow{f} Y \xlongrightarrow{g} Z
\end{align*}
in $\mathcal{C}$. This gives a projective resolution
\begin{align*}
0 \longrightarrow \Hom_{\mathcal{C}}(Z,-)\longrightarrow\Hom_{\mathcal{C}}(Y,-)\longrightarrow    \Hom_{\mathcal{C}}(X,-) \longrightarrow F \longrightarrow 0
\end{align*}
of $F$ in $\textnormal{fp}(\mathcal{C}, \Ab)$. Applying $\Hom_{\textnormal{fp}(\mathcal{C},\Ab)}(-,\Hom_{\mathcal{C}}(C,-))$ yields
\begin{align*}
    0 &\longrightarrow \Hom_{\textnormal{fp}(\mathcal{C},\Ab)}(F,\Hom_{\mathcal{C}}(C,-))\\
    & \longrightarrow \Hom_{\mathcal{C}}(C,X) \longrightarrow \Hom_{\mathcal{C}}(C,Y) \longrightarrow \Hom_{\mathcal{C}}(C,Z). 
\end{align*}
By exactness of 
\begin{align*}
0 \longrightarrow \Hom_{\mathcal{C}}(C,X) \longrightarrow \Hom_{\mathcal{C}}(C,Y) \longrightarrow \Hom_{\mathcal{C}}(C,Z),
\end{align*}
the equality $(\ast)$ follows. Thus, $\Hom_\mathcal{C}(C,-)$ is $\mathcal{S}_\mathcal{E}$-closed in $\textnormal{fp}(\mathcal{C}, \Ab)$ and the functor is fully faithful by \cite[Lemma 2.2.5]{Krause4}.

Consider a short exact sequence
\begin{align*}
    0 \longrightarrow \Hom_{\mathcal{C}}(C,-) \longrightarrow F \longrightarrow \Hom_{\mathcal{C}}(X,-) \longrightarrow 0
\end{align*}
in $\textnormal{fp}(\mathcal{C}, \Ab)/\mathcal{S}_\mathcal{E}$ with $C,X \in \mathcal{C}$. There exists an epimorphism $\Hom_{\mathcal{C}}(Y,-) \rightarrow F$, which gives a commutative diagram
\begin{equation*}
\begin{tikzcd}
0 \arrow[r]  & K \arrow[r] \arrow[d] & \Hom_{\mathcal{C}}(Y,-) \arrow[r, "\eta"] \arrow[d]& \Hom_{\mathcal{C}}(X,-) \arrow[d, equals] \arrow[r] & 0 \\
        0 \arrow[r] &\Hom_{\mathcal{C}}(C,-) \arrow[r] & F \arrow[r]& \Hom_{\mathcal{C}}(X,-) \arrow[r] & 0
\end{tikzcd} \tag{$\ast \ast$}
\end{equation*}
in $\textnormal{fp}(\mathcal{C}, \Ab)/\mathcal{S}_\mathcal{E}$ with exact rows, where $K$ is the kernel of $\eta$. In particular, the left square is a pushout diagram. By fully faithfulness, $\eta$ corresponds to some $f \colon X \rightarrow Y$ with $\coker \Hom_{\mathcal{C}}(f,-) \in \mathcal{S}_\mathcal{E}$. By Theorem \ref{eno} it follows that $f$ is an $\mathcal{E}$-monomorphism. Let $g\colon Y \rightarrow Z$ be the cokernel of $f$. Then $K \cong \Hom_{\mathcal{C}}(Z,-)$ and $K\rightarrow \Hom_{\mathcal{C}}(C,-)$ corresponds to some $h \colon C \rightarrow Z$. Taking the pullback $P$ of $g$ along $h$ yields a commutative diagram
\begin{equation*}
\begin{tikzcd}
    X \arrow[r, "f"] \arrow[d, equals] &Y \arrow[r, "g"] &  Z \\
    X \arrow[r] &P \arrow[u] \arrow[r] &C \arrow[u, "h", swap]
\end{tikzcd}
\end{equation*}
with $\mathcal{E}$-exact rows (see Section 1.1). Comparing the image of this diagram under the functor $\mathcal{C} \rightarrow \textnormal{fp}(\mathcal{C}, \Ab)^\op /\mathcal{S}_\mathcal{E}$ with $(\ast \ast)$, it follows that $F \cong \Hom_{\mathcal{C}}(P,-)$. Hence, the essential image is closed under extensions.
\end{proof}

\begin{rem}\label{induce}\rm It is well-known that a full additive extension-closed subcategory of an abelian category inherits an exact structure from the abelian structure by considering all sequences in the subcategory, which are short exact in the abelian category. See for example \cite[Lemma 10.20]{Buehler}. Thus, by Lemma \ref{small} the category $\mathcal{C}$ inherits an exact structure $\mathcal{E}'$ from the abelian structure of $\textnormal{fp}(\mathcal{C}, \Ab)^\op /\mathcal{S}_\mathcal{E}$. More precisely, a sequence\begin{align*}
    X\xlongrightarrow{f} Y \xlongrightarrow{g} Z
\end{align*}
is $\mathcal{E}'$-exact in $\mathcal{C}$ if and only if 
\begin{align*}
    0 \longrightarrow \Hom_{\mathcal{C}}(Z,-) \longrightarrow \Hom_{\mathcal{C}}(Y,-) \longrightarrow \Hom_{\mathcal{C}}(X,-) \longrightarrow 0
\end{align*}
is exact in $\textnormal{fp}(\mathcal{C}, \Ab) /\mathcal{S}_\mathcal{E}$. This is the case if and only if $\coker \Hom_{\mathcal{C}}(f,-) \in \mathcal{S}_\mathcal{E}$, which is the case if and only if $f$ is an $\mathcal{E}$-monomorphism by Theorem \ref{eno}. It follows that $\mathcal{E}' = \mathcal{E}$.
\end{rem}

Next, we will describe the embedding of $\mathcal{A}$ into its relative purity category $\mathbf{P}_\mathcal{E}(\mathcal{A})$. This gives us a big version of Lemma \ref{small} and generalizes Theorem 1.4.

\begin{thm}\label{big} Let $\mathcal{A}$ be a locally finitely presented category with products and let $\mathcal{E}$ be an exact structure on $\mathcal{C} = \fp \mathcal{A}$. There exists a fully faithful functor
\begin{align*}
    \textnormal{ev}_\mathcal{E} \colon \mathcal{A} \longrightarrow \mathbf{P}_\mathcal{E}(\mathcal{A}),\quad X \mapsto \bar{X},
\end{align*}
which commutes with filtered colimits and cokernels. Moreover, the essential image of $\textnormal{ev}_\mathcal{E}$ is closed under extensions.
\end{thm}

\begin{proof} The functor $\textnormal{ev}_\mathcal{E}$ equals the composition of the functors in the second row of the commutative diagram
\begin{equation*}
    \begin{tikzcd}
        \mathcal{C} \arrow[r] \arrow[d]& \textnormal{fp} (\mathcal{C}, \Ab)^\op \arrow[r] \arrow[d] & \textnormal{fp} (\mathcal{C}, \Ab)^\op/\mathcal{S}_\mathcal{E} \arrow[d, "\varphi"]\\
        \mathcal{A} \arrow[r, "\textnormal{ev}"] & \mathbf{P}(\mathcal{A}) \arrow[r, "q"] & \mathbf{P}_\mathcal{E}(\mathcal{A}).
    \end{tikzcd}
\end{equation*}
By Theorem 1.4 and the properties of the quotient functor $q$, the functor $\textnormal{ev}_\mathcal{E}$ commutes with filtered colimits and cokernels. 

Recall that $\varphi$ induces an equivalence between $\textnormal{fp}(\mathcal{C}, \Ab)^\textnormal{op}/\mathcal{S}_\mathcal{E}$ and $\fp \mathbf{P}_\mathcal{E}(\mathcal{A})$. Further, the composition of the two functors in the first row equals the functor in Lemma \ref{small}. Since it is fully faithful and $\textnormal{ev}_\mathcal{E}$ commutes with filtered colimits, it follows that
\begin{align*}
\Hom_{\mathbf{P}_\mathcal{E}(\mathcal{A})}(\bar{X},\bar{Y})    &= \Hom_{\mathbf{P}_\mathcal{E}(\mathcal{A})}(\varinjlim \bar{X_i}, \varinjlim \bar{Y_j})\\ &\cong \varprojlim \varinjlim \Hom_{ \fp \mathbf{P}_\mathcal{E}(\mathcal{A})}(\bar{X_i}, \bar{Y_j}) \\
    &\cong \varprojlim \varinjlim \Hom_{\mathcal{C}}(X_i, Y_j)\\
    &\cong \Hom_{\mathcal{A}}(\varinjlim X_i, \varinjlim Y_j)\\
    &= \Hom_{\mathcal{A}}(X,Y)
\end{align*}
for $X,Y \in \mathcal{A}$ with $X = \varinjlim X_i, Y = \varinjlim Y_j$ and $X_i, Y_j \in \mathcal{C}$. Thus, $\textnormal{ev}_\mathcal{E}$ is also fully faithful. 

By Lemma \ref{small} the essential image of $\textnormal{ev}_\mathcal{E}$ restricted to $\mathcal{C}$ is closed under extensions in $\mathbf{P}_\mathcal{E}(\mathcal{A})$. Because $\textnormal{ev}_\mathcal{E}$ commutes with filtered colimits, also the essential image of $\textnormal{ev}_\mathcal{E}$ is closed under extensions in $\mathcal{A}$ by Lemma \ref{ext}.
\end{proof}

\begin{lem}\label{ext} Let $\mathcal{B}$ be a locally coherent category and $\mathcal{D}$ a full additive subcategory of $\fp \mathcal{B}$ closed under extensions. Then $\varinjlim \mathcal{D}$ is extension-closed in $\mathcal{B}$.
\end{lem}

\begin{proof} We have $X\in \varinjlim \mathcal{D}$ if and only if every morphism $C \rightarrow X$ with $C\in \fp \mathcal{B}$ factors through some $D\in \mathcal{D}$ by \cite[Lemma 11.1.6]{Krause4}. Consider a short exact sequence
\begin{align*}
    0 \longrightarrow X \longrightarrow Y \longrightarrow Z \longrightarrow 0
\end{align*}
in $\mathcal{B}$ with $X,Z \in \varinjlim \mathcal{D}$. Let $C \rightarrow Y$ be arbitrary with $C\in \fp \mathcal{B}$. The composition $C\rightarrow Y \rightarrow Z$ factors through some $D\in \mathcal{D}$. Taking a pullback, we obtain a commutative diagram
\begin{equation*}
    \begin{tikzcd}
        0 \arrow[r] & X \arrow[d, equals] \arrow[r] & Y  \arrow[r] & Z \arrow[r] & 0 \\
        0 \arrow[r] & X \arrow[r] & P \arrow[u] \arrow[r] & D \arrow[r] \arrow[u] & 0 \\
        & & C \arrow[u] & & 
    \end{tikzcd}
\end{equation*}
with exact rows. Consider a  morphism $E \rightarrow P$ with $E\in \fp \mathcal{B}$ such that the composition $E \rightarrow P \rightarrow D$ is an epimorphism and let $C\oplus E \rightarrow P$ be the induced morphism. This yields a commutative diagram
\begin{equation*}
    \begin{tikzcd}
0 \arrow[r] & X \arrow[r] & P \arrow[r] & D \arrow[r] & 0\\
0 \arrow[r] & K \arrow[u] \arrow[r] & E\oplus C \arrow[r] \arrow[u] & D \arrow[u, equals] \arrow[r] & 0
    \end{tikzcd}
\end{equation*}
with exact rows and $K \in \fp \mathcal{B}$. Now $K\rightarrow X$ factors through some $D'\in \mathcal{D}$. Taking a pushout, we obtain a commutative diagram 
\begin{equation*}
    \begin{tikzcd}
    &&P&&\\
0 \arrow[r] & D' \arrow[r] & P' \arrow[r] \arrow[u] & D \arrow[d, equals] \arrow[r] & 0\\
0 \arrow[r] & K \arrow[u] \arrow[r] & E\oplus C \arrow[r] \arrow[u] & D \arrow[u, equals] \arrow[r] & 0
    \end{tikzcd}
\end{equation*}
with exact rows. Since $\mathcal{D}$ is extension-closed, we have $P'\in \mathcal{D}$. The original morphism $C\rightarrow Y$ equals the composition
\begin{align*}
C \longrightarrow E\oplus C \longrightarrow P' \longrightarrow P \longrightarrow Y,
\end{align*}
where $C\rightarrow E\oplus C$ is the canonical inclusion. Thus, it factors through an object in $\mathcal{D}$ and it follows that $Y\in \varinjlim \mathcal{D}$.
\end{proof}

By Theorem \ref{big} the category $\mathcal{A}$ inherits an exact structure $\bar{\mathcal{E}}$ from the abelian structure of $\mathbf{P}_\mathcal{E}(\mathcal{A})$. More precisely, a sequence
\begin{align*}
    X \longrightarrow Y \longrightarrow Z
\end{align*}
is $\bar{\mathcal{E}}$-exact in $\mathcal{A}$ if and only if 
\begin{align*}
    0 \longrightarrow \bar{X} \longrightarrow \bar{Y} \longrightarrow \bar{Z} \longrightarrow 0 
\end{align*}
is a short exact sequence in $\mathbf{P}_\mathcal{E}(\mathcal{A})$. For example $\bar{\mathcal{E}}_\bot = \mathcal{E}_\textnormal{pure}$ (see Section 1.3). The following generalizes the fact that every pure-exact sequence is a filtered colimit of split exact sequences and offers a second description of $\bar{\mathcal{E}}$. 

\begin{coro}\label{second} The exact structure $\bar{\mathcal{E}}$ equals $\varinjlim \mathcal{E}$, that is every $\bar{\mathcal{E}}$-exact sequence in $\mathcal{A}$ is a filtered colimit of $\mathcal{E}$-exact sequences in $\mathcal{C}$.
\end{coro}

\begin{proof} An $\bar{\mathcal{E}}$-exact sequence $X \rightarrow Y \rightarrow Z$ in $\mathcal{A}$ yields a short exact sequence 
\begin{align*}
    0 \longrightarrow \bar{X} \longrightarrow \bar{Y} \longrightarrow \bar{Z} \longrightarrow 0 
\end{align*}
in $\mathbf{P}_\mathcal{E}(\mathcal{A})$. The functor $\textnormal{ev}_\mathcal{E}$ commutes with filtered colimits and is fully faithful by Theorem \ref{big}. Since $\mathcal{A} = \varinjlim \mathcal{C}$, it follows from Lemma \ref{ext} that the above short exact sequence is a filtered colimit of short exact sequences
\begin{align*}
    0 \longrightarrow \bar{X'} \longrightarrow \bar{Y'} \longrightarrow \bar{Z'} \longrightarrow 0 
\end{align*}
in $\fp \mathbf{P}_\mathcal{E}(\mathcal{A})$ with $X',Y', Z' \in \mathcal{C}$. Because $\fp \mathbf{P}_\mathcal{E}(\mathcal{A})$ identifies with $\textnormal{fp}(\mathcal{C}, \Ab)^\op/\mathcal{S}_\mathcal{E}$, the sequence $X' \rightarrow Y' \rightarrow Z'$ is $\mathcal{E}$-exact by Remark \ref{induce}. By fully faithfulness, the sequence $X\rightarrow Y\rightarrow Z$ is a filtered colimit of the sequences $X' \rightarrow Y' \rightarrow Z'$.
\end{proof}

Corollary \ref{second} shows that one can lift an exact structure $\mathcal{E}$ on the small category $\mathcal{C}$ to an exact structure on the big category $\mathcal{A}$ by taking filtered colimits. This is a special case of \cite[Theorem 2.7]{Positselski} and in this way one obtains a \emph{locally coherent} exact structure as in \cite{Positselski}. Positselski proves the existence of enough {$\mathcal{E}'$-injectives} for every locally coherent exact structure $\mathcal{E}'$ \cite[Corollary 5.4]{Positselski} by showing that one is of \emph{Grothendieck type} as in \cite[Definition 3.11]{stovi}. In our case, we will also show the existence of enough $\bar{\mathcal{E}}$-injectives using a different approach, essentially via the embedding of $\mathcal{A}$ into $\mathbf{P}_\mathcal{E}(\mathcal{A})$ in Theorem \ref{big}.

\begin{lem}\label{trans} The functor $\textnormal{ev}_{\mathcal{E}}\colon \mathcal{A} \longrightarrow \mathbf{P}_\mathcal{E}(\mathcal{A})$ induces an equivalence between 
\begin{itemize}
    \item[\rm (1)] the  fp-$\bar{\mathcal{E}}$-injectives in $\mathcal{A}$ and the fp-injectives in $\mathbf{P}_\mathcal{E}(\mathcal{A})$, as well as
    \item[\rm (2)] the $\bar{\mathcal{E}}$-injectives in $\mathcal{A}$ and the injectives in $\mathbf{P}_\mathcal{E}(\mathcal{A})$.
\end{itemize}
\end{lem}

\begin{proof} Recall that we have a definable subcategory $\mathcal{X}_\mathcal{E}$ of $\mathcal{A}$ and a closed set $\mathcal{U}_\mathcal{E}$ in $\Ind \mathcal{A}$ associated to the exact structure $\mathcal{E}$ on $\mathcal{C}$ by Corollary \ref{enosim} and the correspondences (i)$\leftrightarrow$(iii)$\leftrightarrow$(iv) in Section 1.6. Following the assignments, we have
\begin{align*}
    \mathcal{X}_\mathcal{E} &= \{A \in \mathcal{A} \mid \Hom_{\mathcal{A}}(f,A)\textnormal{ is surjective for all } \mathcal{E}\textnormal{-monomorphisms }f\},\\
    \mathcal{U}_\mathcal{E} &= \Ind \mathcal{A} \cap \mathcal{X}_\mathcal{E}.
\end{align*}
By the correspondence (ii)$\rightarrow$(iii),(iv) it follows that, under $\textnormal{ev}_\mathcal{E}$, the definable subcategory $\mathcal{X}_\mathcal{E}$ identifies with the fp-injectives in $\mathbf{P}_\mathcal{E}(\mathcal{A})$ and the $\mathcal{E}_\textnormal{pure}$-injectives in $\mathcal{X}_\mathcal{E}$ identify with the injectives in $\mathbf{P}_\mathcal{E}(\mathcal{A})$. Thus, for (1) it suffices to show that the fp-$\bar{\mathcal{E}}$-injectives in $\mathcal{A}$ coincide with $\mathcal{X}_\mathcal{E}$ and for (2) that the $\bar{\mathcal{E}}$-injectives in $ \mathcal{A}$ coincide with the $\mathcal{E}_\textnormal{pure}$-injectives in $\mathcal{X}_\mathcal{E}$.

(1) Let $A\in \mathcal{A}$ be fp-$\bar{\mathcal{E}}$-injective and consider an $\mathcal{E}$-exact sequence
\begin{align*}
    0 \longrightarrow X \xlongrightarrow{f} Y \xlongrightarrow{g} Z \longrightarrow 0
\end{align*}
in $\mathcal{C}$. This is also an $\bar{\mathcal{E}}$-exact sequence in $\mathcal{A}$. Hence, the pushout of $f$ along an arbitrary morphism $X \rightarrow A$ yields an $\bar{\mathcal{E}}$-exact sequence $0 \rightarrow A \rightarrow P \rightarrow Z \rightarrow 0$. Now $\Ext_{\bar{\mathcal{E}}}^1(Z,A) = 0$ implies that this sequence splits and the morphism $X \rightarrow A$ factors through $f$. It follows that $\Hom_{\mathcal{A}}(f,A)$ is surjective and $A\in \mathcal{X}_\mathcal{E}$.

Let $A\in \mathcal{X}_\mathcal{E}$. Then $\bar{A}$ is fp-injective in $\mathbf{P}_\mathcal{E} (\mathcal{A})$. Now for $C\in \mathcal{C}$ an $\bar{\mathcal{E}}$-exact sequence $A\rightarrow B \rightarrow C$ yields a short exact sequence
\begin{align*}
    0 \longrightarrow \bar{A} \longrightarrow \bar{B} \longrightarrow \bar{C} \longrightarrow 0
\end{align*}
in $\mathbf{P}_\mathcal{E}(\mathcal{A})$ with $\bar{C}\in \fp \mathbf{P}_\mathcal{E}(\mathcal{A})$. Because $\bar{A}$ is fp-injective, this sequence must split. By Theorem \ref{big} the functor $\textnormal{ev}_\mathcal{E}$ is fully faithful, so also $A \rightarrow B \rightarrow C$ splits. Hence, $A$ is fp-$\bar{\mathcal{E}}$-injective.

(2) Let $A\in \mathcal{A}$ be $\bar{\mathcal{E}}$-injective. Then $A$ is also fp-$\bar{\mathcal{E}}$-injective, so $A\in \mathcal{X}_\mathcal{E}$. Further, $\mathcal{E}_\bot \subseteq \mathcal{E}$ implies $\mathcal{E}_\textnormal{pure} \subseteq \bar{\mathcal{E}}$. Hence, $A$ is also $\mathcal{E}_\textnormal{pure}$-injective.

Let $A\in \mathcal{X}_\mathcal{E}$ be $\mathcal{E}_\textnormal{pure}$-injective. Then $\bar{A}$ is injective in $\mathbf{P}_\mathcal{E}(\mathcal{A})$. Now an $\bar{\mathcal{E}}$-exact sequence $A\rightarrow B \rightarrow C$ yields a short exact sequence
\begin{align*}
    0 \longrightarrow \bar{A} \longrightarrow \bar{B} \longrightarrow \bar{C} \longrightarrow 0
\end{align*}
in $\mathbf{P}_\mathcal{E}(\mathcal{A})$. Because $\bar{A}$ is injective, this sequence must split. By Theorem \ref{big} the functor $\textnormal{ev}_\mathcal{E}$ is fully faithful, so also $A\rightarrow B \rightarrow C$ splits. Hence, $A$ is \mbox{$\bar{\mathcal{E}}$-injective}.
\end{proof}

The proof of Lemma \ref{trans} implies the following. 

\begin{coro}\label{int} \begin{itemize}
    \item[\rm (1)] The fp-$\bar{\mathcal{E}}$-injectives form the definable subcategory $\mathcal{X}_\mathcal{E}$.
    \item[\rm (2)] The indecomposable $\bar{\mathcal{E}}$-injectives form the closed set $\mathcal{U}_\mathcal{E}$ in $\Ind \mathcal{A}$.
\end{itemize}
\end{coro}


Before we can show the existence of enough $\bar{\mathcal{E}}$-injectives in $\mathcal{A}$, we need the following lemma. 

\begin{lem}\label{monough} A morphism $X\rightarrow Y$ in $\mathcal{A}$ is an $\bar{\mathcal{E}}$-monomorphism if and only if $\bar{X} \rightarrow \bar{Y}$ is a monomorphism in $\mathbf{P}_\mathcal{E}(\mathcal{A})$.
\end{lem}
\begin{proof} One implication is clear. Let $f\colon X \rightarrow Y$ and assume that $\bar{f}\colon \bar{X} \rightarrow \bar{Y}$ is a monomorphism in $\mathbf{P}_\mathcal{E}(\mathcal{A})$. We will show that for all $f'\colon X'\rightarrow Y'$ in $\mathcal{C}$ every morphism of morphisms
\begin{equation*}
    \begin{tikzcd}
        X' \arrow[d] \arrow[r,"f'"] & Y' \arrow[d]\\
        X \arrow[r, "f"] & Y
    \end{tikzcd} \tag{$\ast$}
\end{equation*}
from $f'$ to $f$ factors through an $\mathcal{E}$-monomorphism. This would imply that $f$ is a direct limit of \mbox{$\mathcal{E}$-monomorphisms} by \cite[Lemma 11.1.6]{Krause4} and thus an $\bar{\mathcal{E}}$-monomorphism by Corollary \ref{second}.

Since $\bar{f}$ is a monomorphism, we have $\ker \bar{f} \in \mathcal{T}_\mathcal{E}$. Recall that $\mathcal{T}_\mathcal{E} = \varinjlim \mathcal{S}$ by the assignment (i)$\leftrightarrow$(ii) in Section 1.6, where $\mathcal{S}$ equals the Serre subcategory of $\fp \mathbf{P}(\mathcal{A})$ corresponding to the Serre subcategory $\mathcal{S}_\mathcal{E}$ of $\textnormal{fp}(\mathcal{C}, \Ab)$ under the identification of $\fp \mathbf{P}(\mathcal{A})$ with $\textnormal{fp}(\mathcal{C}, \Ab)^\op$ (see Section 1.4). By the description of $\mathcal{S}_\mathcal{E}$, we have
\begin{align*}
    \mathcal{S} = \{F\cong \ker \bar{g} \mid g\textnormal{ is an }\mathcal{E}\textnormal{-monomorphism} \}.
\end{align*}
Now the diagram $(\ast)$ induces a morphism $\ker \bar{f'} \rightarrow \ker \bar{f}$, which must factor through some $\ker \bar{g}$ for an $\mathcal{E}$-monomorphism $g \colon C \rightarrow D$, since $\mathcal{T}_\mathcal{E} = \varinjlim \mathcal{S}$. Because for all $A\in \mathcal{A}$ the object $\bar{A}$ is fp-injective in $\mathbf{P}(\mathcal{A})$ by Theorem 1.4, we obtain a commutative diagram
\begin{equation*}
    \begin{tikzcd}
        0 \arrow[r] & \ker \bar{f'} \arrow[d] \arrow[r] & \bar{X'} \arrow[d] \arrow[r] & \bar{Y'} \arrow[d]\\
        0 \arrow[r] & \ker \bar{g} \arrow[d] \arrow[r] & \bar{C} \arrow[d] \arrow[r] & \bar{D} \arrow[d] \\
        0 \arrow[r] & \ker \bar{f} \arrow[r] & \bar{X} \arrow[r] & \bar{Y} 
    \end{tikzcd}
\end{equation*}
with exact rows. Since the functor $\textnormal{ev}$ is fully faithful by Theorem 1.4, this yields a second morphism of morphisms 
\begin{equation*}
    \begin{tikzcd}
        X' \arrow[d] \arrow[r,"f'"] & Y' \arrow[d]\\
        X \arrow[r, "f"] & Y,
    \end{tikzcd} \tag{$\ast \ast$}
\end{equation*}
which factors through an $\mathcal{E}$-monomorphism. It is left to show that the difference between $(\ast \ast)$ and $(\ast)$, denoted by $\Delta$, factors through an $\mathcal{E}$-monomorphism. This difference induces the zero morphism between $\ker \bar{f}$ and $\ker \bar{f'}$. Corresponding to $\Delta$, consider the commutative diagram
\begin{equation*}
    \begin{tikzcd}
        0 \arrow[r] & \ker \bar{f'} \arrow[d, dotted] \arrow[r] & \bar{X'} \arrow[d, "\bar{\alpha}", swap] \arrow[r, "\bar{f'}"] & \bar{Y'} \arrow[d, "\bar{\beta}"]\\
        0 \arrow[r] & \ker \bar{f} \arrow[r] & \bar{X} \arrow[r, "\bar{f}"] & \bar{Y}. 
    \end{tikzcd}
\end{equation*}
Then $\bar{\alpha}$ factors through the image of $\bar{f'}$. Because $\bar{X}$ is fp-injective, $\bar{\alpha}$ also factors through $\bar{f'}$. This yields $\delta \colon Y' \rightarrow X$ with $\alpha = \delta f'$. Hence, $\Delta$ factors as
\begin{equation*}
    \begin{tikzcd}
        X' \arrow[d, "\alpha", swap] \arrow[r,"f'"] & Y' \arrow[d, "(\delta\, \beta)^\intercal"]\\
        X \arrow[d, equals] \arrow[r,"(1_X \, f)^\intercal"] & X\oplus Y \arrow[d, "(0\, 1_Y)"]\\
        X \arrow[r, "f"] & Y.
    \end{tikzcd}
\end{equation*}
Note that the middle row is a split monomorphism. It equals a filtered colimit of split monomorphisms in $\mathcal{C}$. Since $X',Y' \in \mathcal{C}$, it follows that $\Delta$ factors through such a split monomorphism. In particular, it factors through an $\mathcal{E}$-monomorphism.   
\end{proof} 

\begin{prop}\label{enough} \begin{itemize}
    \item[\rm (1)] Every $X\in \mathcal{A}$ admits an $\bar{\mathcal{E}}$-injective envelope $X\rightarrow Q$. 
    \item[\rm (2)] For all $X\in \mathcal{A}$ there exists a product $Q$ of indecomposable $\bar{\mathcal{E}}$-injective objects in $\mathcal{A}$ and an $\bar{\mathcal{E}}$-monomorphism $X \rightarrow Q$.
\end{itemize}
\end{prop}

\begin{proof} (1) Let $X \in \mathcal{A}$. Since $\mathbf{P}_\mathcal{E}(\mathcal{A})$ is a Grothendieck category, there exists an injective envelope $\bar{X} \rightarrow F$ in $\mathbf{P}_\mathcal{E}(\mathcal{A})$. By Lemma \ref{trans} we can choose $F = \bar{Q}$ for an $\bar{\mathcal{E}}$-injective object $Q\in \mathcal{A}$. Because $\textnormal{ev}_\mathcal{E}$ is fully faithful by Theorem \ref{big},  the morphism $\bar{X} \rightarrow \bar{Q}$ corresponds to a left minimal morphism $X\rightarrow Q$, which is also an $\bar{\mathcal{E}}$-monomorphism by Lemma \ref{monough}. Thus, $X\rightarrow Q$ is an $\bar{\mathcal{E}}$-injective envelope. 

(2) By Lemma \ref{trans} the functor ${\textnormal{ev}_\mathcal{E}} \colon \mathcal{A} \to \mathcal{P}_\mathcal{E}(\mathcal{A})$
induces an equivalence between the subcategory of $\bar{\mathcal{E}}$-injective objects in $\mathcal{A}$ and the subcategory of injective objects in $\mathcal{P}_\mathcal{E}(\mathcal{A})$, which is closed under products. Thus, $\textnormal{ev}_\mathcal{E}$ restricted to the subcategory of $\bar{\mathcal{E}}$-injective objects commutes with products.

Let $X\in \mathcal{A}$. By \cite[Lemma 3.1]{Krause2} there exists a monomorphism $\bar{X} \rightarrow F$, where $F$ is a product of indecomposable injective objects in $\mathbf{P}_\mathcal{E}(\mathcal{A})$. By the above, we can choose $F = \bar{Q}$, where $Q$ is a product of indecomposable $\bar{\mathcal{E}}$-injective objects in $\mathcal{A}$. As for (1), the morphism $\bar{X} \rightarrow \bar{Q}$ corresponds to an $\bar{\mathcal{E}}$-monomorphism $X\rightarrow Q$.
\end{proof}

We continue by investigating the $\bar{\mathcal{E}}$-projective objects in $\mathcal{A}$. In contrast to the \mbox{$\bar{\mathcal{E}}$-injective} objects, there are not always enough $\bar{\mathcal{E}}$-projectives. The following generalizes the fact that every object in $\mathcal{C}$ is $\mathcal{E}_\textnormal{pure}$-projective, $\mathcal{A}$ has enough \mbox{$\mathcal{E}_\textnormal{pure}$-projectives} and every $\mathcal{E}_\textnormal{pure}$-projective object is a direct summand of an arbitrary direct sum of objects in $\mathcal{C}$.

\begin{prop} Every $\mathcal{E}$-projective object is also $\bar{\mathcal{E}}$-projective. Moreover, if $\mathcal{C}$ has enough ${\mathcal{E}}$-projectives, then $\mathcal{A}$ has enough \mbox{$\bar{\mathcal{E}}$-projectives} and every $\bar{\mathcal{E}}$-projective object is a direct summand of an arbitrary direct sum of \mbox{$\mathcal{E}$-projective} objects.
\end{prop}

\begin{proof} Let $P$ be $\mathcal{E}$-projective and $f$ an $\bar{\mathcal{E}}$-epimorphism in $\mathcal{A}$. Then $f = \varinjlim f_i$ for $\mathcal{E}$-epimorphisms $f_i$ in $\mathcal{C}$ by Corollary \ref{second}. Now $\Hom_{\mathcal{A}}(P, f) = \varinjlim \Hom_{\mathcal{C}}(P, f_i)$, since $P$ is finitely presented. Moreover, because $P$ is $\mathcal{E}$-projective, it follows that $\Hom_{\mathcal{C}}(P, f_i)$ is surjective and so is $\Hom_{\mathcal{A}}(P, f)$. We conclude that $P$ is $\bar{\mathcal{E}}$-projective.

If $\mathcal{C}$ has enough $\mathcal{E}$-projectives, then for all $C\in \mathcal{C}$ there exists an ${\mathcal{E}}$-epimorphism $\pi$ starting in an ${\mathcal{E}}$-projective object and ending in $C$. Now consider the morphism $f_C \colon \bigoplus P \rightarrow C$, where the direct sum goes over all $\mathcal{E}$-projective objects in $\mathcal{C}$ (up to isomorphism) and all morphisms $P\rightarrow C$. Clearly, $f_C$ is the filtered colimit of all morphisms $g_{C} \colon \bigoplus P \rightarrow C$, where the direct sum goes over a finite collection of $\mathcal{E}$-projective objects in $\mathcal{C}$, containing at least the morphism $\pi$. Since $\pi$ is an ${\mathcal{E}}$-epimorphism, also $g_C$ is an ${\mathcal{E}}$-epimorphism by \cite[Proposition 7.6]{Buehler}. It follows that $f_C$ is an $\bar{\mathcal{E}}$-epimorphism by Corollary \ref{second}.

For $X\in \mathcal{A}$ we have $X = \varinjlim X_i$ with $X_i\in \mathcal{C}$. As before, consider the canonical morphisms $f_X$ and $f_{X_i}$. Since $\Hom_{\mathcal{A}}(P, X) = \varinjlim \Hom_{\mathcal{C}}(P,X_i)$ for $P\in \mathcal{C}$, we also have $f_X = \varinjlim f_{X_i}$ and it follows that $f_X$ is an $\bar{\mathcal{E}}$-epimorphism. Hence, $\mathcal{A}$ has enough $\bar{\mathcal{E}}$-projectives. Further, if $X$ is $\bar{\mathcal{E}}$-projective, then $f_X$ splits and $X$ is a direct summand of a direct sum of $\mathcal{E}$-projective objects.
\end{proof}

\begin{exa}\label{abelian}\rm Let $\mathcal{A}$ be a locally finitely presented abelian category and $\mathcal{C} = \fp \mathcal{A}$. In that case $\mathcal{A}$ is a Grothendieck category and $\mathcal{C}$ is closed under extensions in $\mathcal{A}$ by \cite[SATZ 1.9]{Breitsprecher}. This induces an exact structure on $\mathcal{C}$, which coincides with all kernel-cokernel pairs in $\mathcal{C}$. Hence, it is the largest exact structure $\mathcal{E}_\top$ on $\mathcal{C}$. In this case, the closed sets in $\Ind \mathcal{A}$ that we get from exact structures $\mathcal{E}$ on $\mathcal{C}$ (see Corollary \ref{int}) are exactly those that contain all indecomposable injectives in $\mathcal{A}$.

    Indeed, if $\mathcal{E}$ is an exact structure on $\mathcal{C}$, then clearly every injective object in $\mathcal{A}$ is also $\bar{\mathcal{E}}$-injective. On the other hand, if $\mathcal{U}$ is a closed set in $\Ind \mathcal{A}$ containing all indecomposable injectives in $\mathcal{A}$, then the corresponding Serre subcategory $\mathcal{S}$ is given by 
    \begin{align*}
        \mathcal{S} = \{F\in \textnormal{fp}(\mathcal{C}, \Ab) \mid \bar{F}(X) = 0\textnormal{ for all }X\in \mathcal{U}\}.
    \end{align*}
    For $F\in \mathcal{S}$ let $F\cong \coker \Hom_{\mathcal{C}}(f,-)$ with $f\colon C\rightarrow D$ in $\mathcal{C}$. Then ${\bar{F}(Q) = 0}$, or equivalently $\Hom_{\mathcal{A}}(f,Q)$ is surjective, for all indecomposable injective objects $Q$ in $\mathcal{A}$ implies that $f$ is a monomorphism by \cite[Lemma 3.1]{Krause2}. Hence, $\mathcal{S} \subseteq \mathcal{S}_{\mathcal{E}_\top}$ and there exists a corresponding exact structure on $\mathcal{C}$ by Corollary \ref{enosim}.
\end{exa}

\section{Translation of properties}

In Section 2 we have seen how to enter the theory of purity of a locally finitely presented category with products from the viewpoint of exact structures on the full subcategory of finitely presented objects. The following theorem summarizes this.

\begin{thm}\label{sum} Let $\mathcal{A}$ be a locally finitely presented category with products and $\mathcal{E}_\top$ the largest exact structure on $\mathcal{C} = \fp \mathcal{A}$. There exist one to one correspondences between the following.
\begin{itemize}
    \item[(1)] Exact structures $\mathcal{E}$ on $\mathcal{C}$.
    \item[(2)] Definable subcategories $\mathcal{X}$ of $\mathcal{A}$ containing all fp-$\bar{\mathcal{E}}_\top$-injectives.
    \item[(3)] Closed sets $\mathcal{U}$ in $\Ind \mathcal{A}$ containing all indecomposable $\bar{\mathcal{E}}_\top$-injectives. 
\end{itemize}
The assignments are given by $\mathcal{E} \mapsto \mathcal{X}_\mathcal{E}$ and $\mathcal{E}\mapsto \mathcal{U}_\mathcal{E}$, where $\mathcal{X}_\mathcal{E}$ denotes the collection of all fp-$\bar{\mathcal{E}}$-injectives and $\mathcal{U}_\mathcal{E}$ the collection of all indecomposable $\bar{\mathcal{E}}$-injectives in $\mathcal{A}$.
\end{thm}
\begin{proof} By Corollary \ref{int} the definition of $\mathcal{U}_\mathcal{E}$ and $\mathcal{X}_\mathcal{E}$ is compatible with the one in Section 2. Now the theorem follows directly by the correspondences in Section 1.6 and in Corollary \ref{enosim}.
\end{proof}

We fix a locally finitely presented category $\mathcal{A}$ with products and set $\mathcal{C} = \fp \mathcal{A}$. Theorem \ref{sum} shows, the bigger the largest exact structure $\mathcal{E}_\top$ on $\mathcal{C}$, the stronger the connection between exact structures and purity. For example, if $\mathcal{A}$ is abelian, then $\mathcal{E}_\top$ is big. In this case $\mathcal{E}_\top$ equals all kernel-cokernel pairs in $\mathcal{C}$ and the collection of indecomposable $\bar{\mathcal{E}}_\top$-injectives coincides with the closure of $\textnormal{Inj}\,\mathcal{A}$ in $\Ind \mathcal{A}$, where $\textnormal{Inj}\,\mathcal{A}$ denotes all indecomposable injectives in $\mathcal{A}$ (see Example \ref{abelian}).

The goal of this section will be to translate properties of exact structures on $\mathcal{C}$ to properties occurring in the theory of purity of $\mathcal{A}$ and vice versa. We fix an exact structure $\mathcal{E}$ on $\mathcal{C}$ as well as the corresponding definable subcategory $\mathcal{X}_\mathcal{E}$ of $\mathcal{A}$ and closed set $\mathcal{U}_\mathcal{E}$ in $\Ind \mathcal{A}$. The exact structure $\mathcal{E}$ is called \emph{finitely generated} if there exists a kernel-cokernel pair $(f,g)$ in $\mathcal{C}$ such that $\mathcal{E}$ is the smallest exact structure containing $(f,g)$.

\begin{prop}\label{quasi} The exact structure ${\mathcal{E}}$ on $\mathcal{C}$ is finitely generated if and only if \mbox{$\Ind \mathcal{A} \setminus \mathcal{U}_\mathcal{E}$} is quasi-compact.
\end{prop}

\begin{proof} For a closed set $\mathcal{U}$ in $\Ind \mathcal{A}$ let $\mathcal{U}^\mathsf{c} = \Ind \mathcal{A}\setminus \mathcal{U}$. For an $\mathcal{E}$-monomorphism $f$ let $\mathcal{E}_f$ be the smallest exact structure on $\mathcal{C}$ containing $(f, \coker f)$. 

Let $\mathcal{E}$ be finitely generated and $\mathcal{U}_\mathcal{E}^\mathsf{c} = \bigcup \mathcal{U}^\mathsf{c}$ an open cover of $\mathcal{U}_\mathcal{E}^\mathsf{c}$, where the union runs over a family $\mathcal{F}$ of closed sets in $\Ind \mathcal{A}$. Without loss of generality we can assume that $\mathcal{F}$ is directed, that is for all $\mathcal{U}$ and $ \mathcal{U}'$ in $ \mathcal{F}$ we have $\mathcal{U} \cup \mathcal{U}' \in \mathcal{F}$. By Theorem \ref{sum} for every $\mathcal{U}\in \mathcal{F}$ there exists a corresponding exact structure $\mathcal{E}_\mathcal{U}$ on $\mathcal{C}$. This yields $\mathcal{E} = \bigcup \mathcal{E}_\mathcal{U}$. Since $\mathcal{E}$ is finitely generated, there is a kernel-cokernel pair $(f,g)$ in $\mathcal{E}$ such that $\mathcal{E}$ is the smallest exact structure containing $(f,g)$. Now $(f,g)$ must be contained in some $\mathcal{E}_{\mathcal{U}}$ and thus $\mathcal{E} = \mathcal{E}_{\mathcal{U}}$. It follows that $\mathcal{U}_\mathcal{E} = \mathcal{U}$ and so $\mathcal{U}_\mathcal{E}^\mathsf{c}$ is quasi-compact.

Let $\Ind \mathcal{A} \setminus \mathcal{U}_\mathcal{E}$ be quasi-compact. We have $\mathcal{E} = \bigcup \mathcal{E}_f$, where the union goes over all $\mathcal{E}$-monomorphisms $f$ in $\mathcal{C}$. The union is directed, since for $\mathcal{E}$-monomorphisms $f$ and $g$, also $f\oplus g$ is an \mbox{$\mathcal{E}$-monomorphism} by \cite[Proposition 2.9]{Buehler} and the inclusion $\mathcal{E}_f, \mathcal{E}_g \subseteq \mathcal{E}_{f\oplus g}$ holds by \cite[Corollary 2.18]{Buehler}. By Theorem \ref{sum} we obtain a directed union $\mathcal{U}_{\mathcal{E}}^\mathsf{c} = \bigcup \mathcal{U}_{\mathcal{E}_f}^\mathsf{c}$. Because $\mathcal{U}_\mathcal{E}^\mathsf{c}$ is quasi-compact, there exists an $\mathcal{E}$-monomorphism $f$ with $\mathcal{U}_{\mathcal{E}} = \mathcal{U}_{\mathcal{E}_f}$. It follows that $\mathcal{E} = \mathcal{E}_f$ and so $\mathcal{E}$ is finitely generated.
\end{proof}

Consider a (finite or infinite) sequence
\begin{align*}
    X_0 \longrightarrow X_1 \longrightarrow X_2 \longrightarrow \dots \tag{$\ast$}
\end{align*}
of morphisms in $\mathcal{C}$. For all $i\geq 0$ the sequence induces a morphism $X_0 \rightarrow X_i$, which equals $1_{X_0}$ for $i=0$. Now $(\ast)$ is called an \emph{$\mathcal{E}$-sequence} if the morphism $\varphi$ in some weak pushout diagram
\begin{equation*}
    \begin{tikzcd}
        X_0 \arrow[r] \arrow[d] & X_{i+1} \arrow[d]\\
        X_{i} \arrow[r, "\varphi"] & P
    \end{tikzcd}
\end{equation*}
is not an $\mathcal{E}$-monomorphism for all $i\geq 0$. Note that weak pushout diagrams exist, since $\mathcal{C}$ has weak cokernels. In particular, $X_0 \rightarrow X_{i+1}$ is never an $\mathcal{E}$-monomorphism.

\begin{exa}\rm Let $f_i \colon X_{i}\rightarrow X_{i+1}, i \geq 0$ be a sequence in $\mathcal{C}$ and $g_i =f_{i-1}f_{i-2}\dots f_0$.
\begin{itemize}
    \item[(1)] Let $\mathcal{E}_\bot$ be the split exact structure on $\mathcal{C}$. The sequence $f_i, i\geq 0$ is an $\mathcal{E}_\bot$-sequence if and only if $g_{i}$ does not factor through $g_{i+1}$ for all $i$.
    \item[(2)] Let $\mathcal{C}$ be abelian and $\mathcal{E}_\top$ the largest exact structure on $\mathcal{C}$. The sequence $f_i, i\geq 0$ is an $\mathcal{E}_\top$-sequence if and only if $\ker g_i, i\geq 0$ is strictly increasing. 
\end{itemize}    
\end{exa}

Recall that there exists a Serre subcategory $\mathcal{S}_\mathcal{E}$ of $\Fp(\mathcal{C}, \Ab)$ corresponding to $\mathcal{E}$ by Corollary \ref{enosim}.

\begin{lem}\label{seq} Let $f_i \colon X_i \rightarrow X_{i+1}, i\geq 0$ be an $\mathcal{E}$-sequence in $\mathcal{C}$ and $g_i = f_{i-1} f_{i-2} ... f_0$. Then
    \begin{align*}
        \coker \Hom_{\mathcal{C}}(g_i,-) \longrightarrow \coker \Hom_{\mathcal{C}}(g_{i-1},-)
    \end{align*}
 defines a chain of proper epimorphisms in $\Fp(\mathcal{C}, \Ab)/\mathcal{S}_{\mathcal{E}}$ between quotients of the functor $\Hom_{\mathcal{C}}(X_0,-)$. Further, every sequence 
    \begin{align*}
        \dots \longrightarrow F_2 \longrightarrow F_1 \longrightarrow  F_0
    \end{align*}
    of proper epimorphisms between quotients of $\Hom_{\mathcal{C}}(X_0,-)$ in $\Fp(\mathcal{C}, \Ab)/\mathcal{S}_{\mathcal{E}}$ is isomorphic to such a chain.
\end{lem}

\begin{proof} Let $q\colon \Fp(\mathcal{C}, \Ab) \rightarrow \Fp(\mathcal{C}, \Ab)/\mathcal{S}_\mathcal{E}$ be the canonical quotient functor. Every epimorphism in $\Fp(\mathcal{C}, \Ab)/\mathcal{S}_\mathcal{E}$ is isomorphic to $q(\alpha)$ for an epimorphism $\alpha \colon F \rightarrow G$ in $\Fp(\mathcal{C}, \Ab)$. We can choose $F = \coker \Hom_{\mathcal{C}}(f,-)$ for a morphism $f \colon X \rightarrow Y$ in $\mathcal{C}$. This yields a commutative diagram
\begin{equation*}
    \begin{tikzcd}
        \Hom_{\mathcal{C}}(Y,-) \arrow[d, "{\Hom_{\mathcal{C}}(h,-)}", swap] \arrow[r, "{\Hom_{\mathcal{C}}(f,-)}"] &[2em] \Hom_{\mathcal{C}}(X,-) \arrow[d, equals] \arrow[r] & F \arrow[r] \arrow[d] & 0 \\
         \Hom_{\mathcal{C}}(Z,-) \arrow[r, "{\Hom_{\mathcal{C}}(g,-)}"] & \Hom_{\mathcal{C}}(X,-) \arrow[r] & G \arrow[r]& 0
    \end{tikzcd}
\end{equation*}
with exact rows and morphisms $g\colon X \rightarrow Z$, $h \colon Z \rightarrow Y$ such that $f = hg$. The kernel $K$ of $\alpha$ equals $\coker \Hom_{\mathcal{C}}(\varphi,-)$, where $\varphi$ occurs in some weak pushout diagram
\begin{equation*}
    \begin{tikzcd}
        X \arrow[r, "f"] \arrow[d, "g", swap] & Y \arrow[d]\\
        Z \arrow[r, "\varphi"] & P.
    \end{tikzcd}
\end{equation*}
Because $q$ is exact, it follows that the kernel of $q(\alpha)$ also equals $K$. Hence, $q(\alpha)$ is a proper epimorphism if and only if $K \notin \mathcal{S}_{\mathcal{E}}$. By Theorem \ref{eno} this is equivalent to $\varphi$ not being an $\mathcal{E}$-monomorphism. The claim now follows inductively by the definition of $\mathcal{E}$-sequences. 
\end{proof}

A locally coherent category $\mathcal{B}$ is \emph{locally noetherian} if every $X\in\fp \mathcal{B}$ is noetherian and $\mathcal{B}$ is \emph{locally finite} if every $X\in \fp\mathcal{B}$ is of finite length. 

\begin{prop}\label{noeth} The following are equivalent.
\begin{itemize}
    \item[\rm (1)] There exists no infinite $\mathcal{E}$-sequence in $\mathcal{C}$.
    \item[\rm (2)] The relative purity category $\mathbf{P}_\mathcal{E}(\mathcal{A})$ is locally noetherian.
    \item[\rm (3)] The subcategory of $\bar{\mathcal{E}}$-injective objects in $\mathcal{A}$ is closed under filtered colimits.
    \item[\rm (4)] The subcategory of $\bar{\mathcal{E}}$-injective objects in $\mathcal{A}$ is closed under coproducts.
    \item[\rm (5)] Every $\bar{\mathcal{E}}$-injective object decomposes into a coproduct of indecomposable objects with local endomorphism rings.
\end{itemize}
\end{prop}

\begin{proof} The category $\fp \mathbf{P}_{\mathcal{E}}(\mathcal{A})$ identifies with $\Fp(\mathcal{C}, \Ab)^\textnormal{op}/\mathcal{S}_\mathcal{E}$ (see Section 1.4) and every object in $\Fp(\mathcal{C}, \Ab)$ is a quotient of $\Hom_{\mathcal{C}}(X,-)$ for some $X\in \mathcal{C}$. Thus, the equivalence between (1) and (2) follows from Lemma \ref{seq}. The $\bar{\mathcal{E}}$-injective objects are identified via $\textnormal{ev}_\mathcal{E}$ with the injective objects in $\mathbf{P}_{\mathcal{E}}(\mathcal{A})$ by Lemma \ref{trans}. Since $\textnormal{ev}_\mathcal{E}$ is fully faithful and commutes with filtered colimits by Theorem \ref{big}, the equivalences between (2), (3), (4) and (5) follow from \cite[Theorem 11.2.12]{Krause2}. 
\end{proof}

An object $X\in \mathcal{A}$ is \emph{endofinite} if for all $C\in \mathcal{C}$ the $\textnormal{End}_\mathcal{A}(X)$-module $\Hom_{\mathcal{A}}(C,X)$ is of finite length.

\begin{prop} The following are equivalent.
\begin{itemize}
    \item[\rm (1)] For every $X\in \mathcal{C}$ the lengths of $\mathcal{E}$-sequences starting in $X$ are bounded.
    \item[\rm (2)] The relative purity category $\mathbf{P}_\mathcal{E}(\mathcal{A})$ is locally finite.
    \item[\rm (3)] Every $\bar {\mathcal{E}}$-injective object in $\mathcal{A}$ is endofinite.
\end{itemize}
\end{prop}

\begin{proof} The proof is the same as the proof of Proposition \ref{noeth} but we make use of \cite[Proposition 13.1.9]{Krause2} instead.
\end{proof}

Recall that a morphism $f\colon X\rightarrow Y$ in an additive category is \emph{left almost split} if $f$ is not a split monomorphism and every morphism starting in $X$ that is not a split monomorphism factors through $f$. In this case $X$ has local endomorphism ring. 

\begin{lem}\label{bigsim} Let $X\in \mathcal{A}$ be $\bar{\mathcal{E}}$-injective. The following are equivalent.
\begin{itemize}
    \item[(1)] The object $X$ is the source of a left almost split morphism in $\mathcal{X}_\mathcal{E}$.
    \item[(2)] The injective envelope of a simple object in $\mathbf{P}_\mathcal{E}(\mathcal{A})$ coincides with $\bar{X}$. 
\end{itemize}
\end{lem}
\begin{proof} The proof is similar to the proof of \cite[Theorem 12.3.13]{Krause4}. Note that $\textnormal{ev}_\mathcal{E}$ identifies the $\bar{\mathcal{E}}$-injectives in $\mathcal{A}$ with the injectives in $\mathbf{P}_{\mathcal{E}}(\mathcal{A})$ as well as $\mathcal{X}_\mathcal{E}$ with the fp-injectives in $\mathbf{P}_{\mathcal{E}}(\mathcal{A})$ by Lemma \ref{trans}. Moreover, the functor $\textnormal{ev}_\mathcal{E}$ is fully faithful by Theorem \ref{big}.

(1)$\Rightarrow$(2): Let $f\colon X \rightarrow Y$ be left almost split in $\mathcal{X}_\mathcal{E}$. If $\ker \bar{f} = 0$, then $\bar{X}\rightarrow \bar{Y}$ is a split monomorphism, since $\bar{X}$ is injective. Then also $f$ is a split monomorphism, which is a contradiction. Hence, $\ker \bar{f} \neq 0$. Let $0 \neq C\subseteq \ker \bar{f}$ be a finitely generated subobject, $U \subseteq C$ a maximal subobject and $C/U \rightarrow \bar{X}'$ an injective envelope in $\mathbf{P}_\mathcal{E}(\mathcal{A})$. Then $C/U$ is simple and $X'$ is indecomposable. Consider the compositions
\begin{align*}
    C \longrightarrow \ker \bar{f} \longrightarrow \bar{X}, \qquad C\longrightarrow C/U \longrightarrow \bar{X}'.
\end{align*}
Because $\bar{X}'$ is injective, there exists a commutative diagram
\begin{equation*}
    \begin{tikzcd}
        \bar{X}' & \\
        C \arrow[r] \arrow[u] & \bar{X}. \arrow[ul]
    \end{tikzcd}
\end{equation*}
Now the morphism $\bar{X} \rightarrow \bar{X}'$ corresponds to a morphism $g\colon X\rightarrow X'$. If $g$ is not a split monomorphism, then $g$ factors through $f$. This contradicts $\bar{f}(C) = 0$ and it follows that $g$ is a split monomorphism. Since ${X}'$ is indecomposable, the morphism $g$ is an isomorphism and $\bar{X}$ is the injective envelope of the simple object $C/U$.

(2)$\Rightarrow$(1): Let $S \rightarrow \bar{X}$ be an injective envelope of a simple object $S$ in $\mathbf{P}_\mathcal{E}(\mathcal{A})$. By \cite[Lemma 11.1.26]{Krause4} and \cite[ Proposition 11.1.27]{Krause4} there exists a morphism $\alpha\colon \bar{X}/S \rightarrow \bar{Y}$ with $Y\in \mathcal{X}_\mathcal{E}$ such that $\Hom_{\mathbf{P}_\mathcal{E}(\mathcal{A})}(\alpha, \bar{A})$ is surjective for $A\in \mathcal{X}_\mathcal{E}$. The induced morphism $\bar{X} \rightarrow \bar{Y}$ corresponds to a morphism $f\colon X\rightarrow Y$. We show that $f$ is left almost split. Clearly, $f$ is not a split monomorphism, since $\bar{f}$ is not a monomorphism. Let $g\colon X\rightarrow Z$ be a morphism with $Z\in \mathcal{X}_\mathcal{E}$ that is not a split monomorphism. By injectivity of $\bar{X}$ it follows that $\bar{g}$ is not a monomorphism and so $\bar{g}(S) = 0$. Thus, $\bar{g}$ factors through $\bar{X} \rightarrow \bar{X}/S$ as well as $\bar{X}/S \rightarrow \bar{Y}$, because $\Hom_{\mathbf{P}_\mathcal{E}(\mathcal{A})}(\alpha, \bar{Z})$ is surjective. Hence, $g$ factors through $f$.  
\end{proof}

A morphism $X \rightarrow Y$ in $\mathcal{C}$ is \emph{almost $\mathcal{E}$-monic} if it is not an $\mathcal{E}$-monomorphism and for every morphism $X\rightarrow Z$, the morphism $\varphi$ in some weak pushout diagram 
\begin{equation*}
    \begin{tikzcd}
        X \arrow[r] \arrow[d]& Y \arrow[d]\\
        Z \arrow[r, "\varphi"] & P
    \end{tikzcd}
\end{equation*}
is an $\mathcal{E}$-monomorphism or the induced morphism $X \rightarrow Y \oplus Z$ is an $\mathcal{E}$-monomorphism. 

\begin{exa}\rm
\begin{itemize} 
    \item[(1)] A morphism $X\rightarrow Y$ in $\mathcal{C}$ is left almost split if and only if it is almost $\mathcal{E}_\bot$-monic and $X$ has local endomorphism ring.    
    \item[(2)] Assume that $\mathcal{C}$ is abelian. A morphism $X\rightarrow Y$ in $\mathcal{C}$ is almost $\mathcal{E}_\top$-monic if and only if its kernel is simple.
\end{itemize}
\end{exa}

The following is a small version of Lemma \ref{bigsim}.

\begin{lem}\label{sim} For $X\in \mathcal{C}$ the following are equivalent.
\begin{itemize}
    \item[(1)] The object $X$ is the source of an almost $\mathcal{E}$-monic morphism in $\mathcal{C}$.
    \item[(2)] There is an epimorphism $\Hom_{\mathcal{C}}(X,-)\rightarrow S$ in $\Fp (\mathcal{C}, \Ab)/{\mathcal{S}_\mathcal{E}}$ with simple $S$. 
\end{itemize}
In this case $S\cong \coker \Hom_\mathcal{C}(f,-)$, where $f$ equals the almost $\mathcal{E}$-monic morphism $f$ starting in $X$.
\end{lem}

\begin{proof} We show that for every morphism $f\colon X\rightarrow Y$ in $\mathcal{C}$ the functor 
\begin{align*}
F= \coker \Hom_{\mathcal{C}}(f,-)    
\end{align*}
is simple in $\Fp(\mathcal{C},\Ab)$ if and only if $f$ is almost $\mathcal{E}$-monic. Then clearly (1) implies (2) and also (2) implies (1), since an epimorphism $\Hom_{\mathcal{C}}(X,-) \rightarrow S$ can be extendend to an exact sequence 
\begin{align*}
    \Hom_{C}(Y,-) \longrightarrow \Hom_{\mathcal{C}}(X,-) \longrightarrow S \longrightarrow 0
\end{align*}
in $\Fp(\mathcal{C},\Ab)/\mathcal{S}_\mathcal{E}$, where the first morphism corresponds to a morphism $X\rightarrow Y$ by Lemma \ref{small}. 

The object $F$ is simple in $\Fp(\mathcal{C},\Ab)/\mathcal{S}_{\mathcal{E}}$ if and only if $F\notin \mathcal{S}_{\mathcal{E}}$ and for every morphism $\Hom_{\mathcal{C}}(Z,-) \rightarrow F$ in $\Fp(\mathcal{C},\Ab)$ the image is contained in $\mathcal{S}_\mathcal{E}$ or the cokernel is contained in $\mathcal{S}_\mathcal{E}$. The morphism $\Hom_{C}(Z,-) \rightarrow F$ corresponds to a morphism $X\rightarrow Z$. The image is isomorphic to $\coker \Hom_{\mathcal{C}}(\varphi,-)$, where $\varphi$ occurs in the pushout diagram
\begin{equation*}
    \begin{tikzcd}
        X \arrow[r] \arrow[d]& Y \arrow[d]\\
        Z \arrow[r, "\varphi"] & P
    \end{tikzcd}
\end{equation*}
and the cokernel is isomorphic to $\coker \Hom_{\mathcal{C}}(g,-)$, where $g$ equals the induced morphism $X \rightarrow Y\oplus Z$. By Theorem \ref{eno} and the definition of an almost $\mathcal{E}$-monic morphism, it now follows that $f$ is almost $\mathcal{E}$-monic if and only if $F$ is simple in $\Fp(\mathcal{C}, \Ab)/\mathcal{S}_\mathcal{E}$.
\end{proof}

The following connects left almost split morphisms in $\mathcal{X}_\mathcal{E}$ and almost $\mathcal{E}$-monic morphisms in $\mathcal{C}$.

\begin{prop}\label{left} Let $Q\in \mathcal{A}$ be the $\bar{\mathcal{E}}$-injective envelope of $C\in \mathcal{C}$. The following are equivalent.
\begin{itemize}
    \item[(1)] There is a left almost split morphism in $\mathcal{X}_\mathcal{E}$ starting in an indecomposable direct summand $X$ of $Q$.
    \item[(2)] There is an almost $\mathcal{E}$-monic morphism starting in $C$.  
\end{itemize}
Moreover, in this case $\{X\}$ is open in $\mathcal{U}_\mathcal{E}$.
\end{prop}

\begin{proof} The proof is similar to \cite[Theorem 3.6]{Krause}.

The $\bar{\mathcal{E}}$-injective envelope corresponds to an injective envelope $\bar{C}\rightarrow \bar{Q}$ in $\mathbf{P}_\mathcal{E}(\mathcal{A})$ by the proof of Proposition \ref{enough}. Simple subobjects of $\bar{Q}$ correspond to simple subobjects of $\bar{C}$, which are automatically finitely presented, since $\mathbf{P}_\mathcal{E}(\mathcal{A})$ is locally coherent. Further, every simple subobject of $\bar{Q}$ corresponds to an indecomposable direct summand of $\bar{Q}$, which is the injective envelope of this simple object. This direct summand is isomorphic to some $\bar{X}$, where $X$ is an indecomposable direct summand of $Q$. Now the equivalence of (1) and (2) follows by Lemma \ref{bigsim} and Lemma \ref{sim} using the identification of $\fp \mathbf{P}_\mathcal{E}(\mathcal{A})$ with $\Fp(\mathcal{C},\Ab)^\textnormal{op}/\mathcal{S}_\mathcal{E}$ (see Section 1.4).

The above shows that, under the assumption of (1) or equivalently (2), there exists an injective envelope $\bar{S} \rightarrow \bar{X}$ in $\mathbf{P}_\mathcal{E}(\mathcal{A})$, where $\bar{S}$ is simple in $\fp \mathbf{P}_\mathcal{E}(\mathcal{A})$. Because $\Hom_{\fp \mathbf{P}_\mathcal{E}(\mathcal{A})}
(\bar{S}, \bar{Y}) = 0$ for every indecomposable injective object $\bar{Y}\not \cong \bar{X}$ in $\mathbf{P}_\mathcal{E}(\mathcal{A})$, it follows that $\{\bar{X}\}$ is open in $\textnormal{Sp}\,\mathbf{P}_\mathcal{E}(\mathcal{A})$ (see Section 1.5). Hence, $\{X\}$ is open in $\mathcal{U}_\mathcal{E}$ by the correspondence (ii)$\rightarrow$(iii),(iv) in Section 1.6.
\end{proof}

\section{The case of an Artin algebra}

Let $k$ be a commutative artinian ring, $A$ an Artin $k$-algebra, $\Mod A$ the category of left $A$-modules and $\mod A$ the full subcategory of finite length modules. The category $\Mod A$ is locally finite with $\fp \Mod A = \mod A$. We set $\Ind A = \Ind \Mod A$ and denote by $\Inj A$, respectively $\textnormal{proj}\, A$, the set of isomorphism classes of indecomposable injective, respectively projective, modules in $\mod A$.

The largest exact structure $\mathcal{E}_\top$ on $\mod A$ coincides with the abelian structure and so $\bar{\mathcal{E}}_\top = \varinjlim \mathcal{E}_\top$ coincides with the abelian structure of $\Mod A$. Moreover, for an indecomposable module $X\in \Mod A$ if $X$  is endofinite, then $X$ is a closed point in $\Ind A$ and if $X$ is of finite length, then $X$ is an isolated point in $\Ind A$, see for example \cite[Theorem 5.1.12]{Prest} and \cite[Corollary 5.3.33]{Prest}. The following result shows that the study of $\Ind \mathcal{A}$ is equivalent to the study of exact structure on $\mod A$. 

\begin{coro}\label{easy} There exists a one to one correspondence 
\begin{equation*}
\left\{\begin{matrix}\text{closed sets}\\
\text{in $\Ind A$}\end{matrix}\right\} 
\longleftrightarrow 
\left\{\begin{matrix}\text{exact structures} \\ \text{on $\mod A$}\end{matrix}\right\} \times \left\{\begin{matrix}\text{subsets} \\ \text{of ${\Inj {A}}$}\end{matrix}\right\}.
\end{equation*}
\end{coro}

\begin{proof} By Theorem \ref{sum} an exact structure $\mathcal{E}$ on $\mod A$ corresponds to a closed set in $\Ind A$ containing $\Inj A$. Now an arbitrary closed set in $\Ind A$ differs by a choice of a subset of $\Inj A$, since $\{X\}$ is closed and open for all $X\in \Inj A$.
\end{proof}

An indecomposable module $X\in \Mod A$ is \emph{generic} if $X$ is endofinite and $X$ is not of finite length. The existence of generic modules is related to the (generalized) second Brauer-Thrall conjecture, which states that if $\mod A$ is of infinite representation type and the simple modules have infinite underlying sets, then there exist infinitely many $n \in \mathbb{N}$ such that there are infinitely many non-isomorphic indecomposable modules of length $n$. Now under the assumptions in the second Brauer-Thrall conjecture, the conjecture is equivalent to the existence of a generic module \cite[Theorem 7.3]{Crawley-Boevey0}.

\begin{prop}\label{generic}{\rm \cite[Proposition 6.23]{Krause}} The assignment
\begin{align*}
    M \mapsto \mathcal{S}_M = \{F\in \Fp(\mod A, \Ab) \mid F(M) = 0\}
\end{align*}
induces a bijection between
\begin{itemize}
    \item[\rm (1)] isomorphism classes of indecomposable endofinite modules $M\in \Mod A$, and
    \item[\rm (2)] maximal Serre subcategories $\mathcal{S}$ of $\Fp(\mod A, \Ab)$ such that $\Fp(\mod A, \Ab)/\mathcal{S}$ has a simple object. 
\end{itemize}
Moreover, $M$ is generic if and only if every simple object in $\Fp(\mod A, \Ab)$ is also contained in $\mathcal{S}_M$.
\end{prop} 

Let $\mathcal{E}$ be an exact structure on $\mod A$. Then $\mathcal{E}$ is \textit{maximal} if the only bigger exact structure is the abelian exact structure of $\mod A$. A short exact sequence
\begin{align*}
    0 \longrightarrow X \xlongrightarrow{f} Y \xlongrightarrow{g} Z \longrightarrow 0
\end{align*}
in $\mod A$ is \emph{almost} $\mathcal{E}$\emph{-exact} if $f$ is almost $\mathcal{E}$-monic. Recall that the above short exact sequence is an \emph{almost split sequence} if $f$ is left almost split and $Z$ is indecomposable. The following translates Proposition \ref{generic} to the language of exact structures.

\begin{thm}\label{genericex} Let $A$ be an Artin algebra. The assignment
\begin{align*}
M \mapsto \mathcal{E}_{M} = \{(f, \coker f) \mid \ker f = 0\text{ and }\coker \Hom_{A}(f,M) = 0\}
\end{align*}
induces a bijection between
\begin{itemize}
    \item[\rm (1)] isomorphism classes of indecomposable endofinite modules $M\in \Mod A$ that are not injective, and
    \item[\rm (2)] maximal exact structures $\mathcal{E}$ on $\mod A$ such that there exists an almost \mbox{$\mathcal{E}$-exact} sequence. 
\end{itemize}
Moreover, $M$ is generic if and only if every almost split sequence is $\mathcal{E}_M$-exact.
\end{thm}

\begin{proof} The exact structure $\mathcal{E}_M$ coincides with the exact structure corresponding to the closed set $\mathcal{U} = \{M\} \cup \Inj A$ from Theorem \ref{sum}. Every bigger exact structure would correspond to a smaller closed set in $\Ind A$ containing $\Inj A$, which can only be $\Inj A$. It follows that $\mathcal{E}_M$ is a maximal exact structure. 

Let $\mathcal{S}_M$ be as in Proposition \ref{generic}. Then there exists a simple object $F$ in $\Fp (\mod A, \Ab)/\mathcal{S}_M$. Further, $F\cong \coker \Hom_{A}(f,-)$ for a morphism $f \colon X \rightarrow Y$ in $\mod A$. If $f= g h$ for a morphism $g$ and an epimorphism $h$ in $\mod A$, then there exists a short exact sequence
 \begin{align*}
         0 \rightarrow \coker \Hom_{A}(g,-) \rightarrow \coker  \Hom_{A}(f,-) \rightarrow \coker \Hom_{A}(h ,-) \rightarrow 0
 \end{align*}
in $\Fp(\mod A, \Ab)$. In such a case, either $F$ is isomorphic to $\coker \Hom_{A}(g,-)$ or isomorphic to $\coker \Hom_{A}(h,-)$ in $\Fp(\mod A, \Ab)/\mathcal{S}_M$, since $F$ is simple. Hence, by induction we may assume that either $f$ is a monomorphism or an epimorphism with simple kernel. In the second case let $f \colon X\rightarrow X/S$ for $S\leq X$ simple, $\iota \colon X\rightarrow I$ an injective hull, $\varphi \colon I \rightarrow I/\iota (S)$ the pushout of $f$ along $\iota$ and $\psi \colon X \rightarrow X/S\oplus I$ induced by $f, \iota$. By construction, there exists a short exact sequence
 \begin{align*}
         0 \rightarrow \coker \Hom_{A}(\varphi,-) \rightarrow \coker  \Hom_{A}(f,-) \rightarrow \coker \Hom_{A}(\psi ,-) \rightarrow 0
 \end{align*}
in $\Fp(\mod A, \Ab)$. Because $M\notin \Inj A$, we have $\coker \Hom_{A}(\varphi, M) = 0$ and thus $F \cong \coker \Hom_{A}(\psi ,-)$ in $\Fp(\mod A, \Ab)/\mathcal{S}_M$. Hence, we may assume that $f$ is a monomorphism. It follows that $\coker \Hom_{A}(f,-)\in \mathcal{S}_{\mathcal{E}_{\top}}$, where $\mathcal{S}_{\mathcal{E}_{\top}}$ is the Serre subcategory of $\Fp(\mathcal{C},\Ab)$ corresponding to the largest exact structure $\mathcal{E}_\top$ on $\mod A$ from Corollary \ref{enosim}. Thus, $F$ is simple in $\mathcal{S}_{\mathcal{E}_{\top}}/(\mathcal{S}_M \cap \mathcal{S}_{\mathcal{E}_{\top}})$. By definition $\mathcal{S}_{\mathcal{E}_M} = \mathcal{S}_M \cap \mathcal{S}_{\mathcal{E}_{\top}}$ and it follows that $F$ is simple in $\Fp(\mod A, \Ab)/\mathcal{S}_{\mathcal{E}_M}$. Now the morphism $f$ is almost $\mathcal{E}_M$-monic by Lemma \ref{sim}. Since $f$ is a monomorphism, the pair $(f, \coker f)$ is an almost $\mathcal{E}_M$-exact sequence.  

For an almost split sequence
\begin{align*}
    0 \longrightarrow X \xlongrightarrow{f} Y \longrightarrow Z \longrightarrow 0
\end{align*}
in $\mod A$ we have $\coker\Hom_{A}(f,M) = 0$ if and only if $X \cong M$. Thus, such a sequence is always $\mathcal{E}_M$-exact if and only if $M$ is generic. 


It is left to check the surjectivity of the assignment. Let $\mathcal{E}$ be a maximal exact structure on $\mod A$ such that there exists an almost $\mathcal{E}$-exact sequence $(f,\coker f)$. Let $\mathcal{U}_{\mathcal{E}}$ be the corresponding closed set in $\Ind A$ by Theorem \ref{sum}. Then $\mathcal{U}_{\mathcal{E}}$ is minimal with the property of containing $\Inj A$ but not being equal to $\Inj A$. It follows that $\mathcal{U} = \mathcal{U}_{\mathcal{E}}\setminus \Inj A$ is a minimal non-empty closed set in $\Ind A$. Hence, the Serre subcategory $\mathcal{S}$ corresponding to $\mathcal{U}$ is maximal (see Section 1.6). Further, the equality $\mathcal{U}_\mathcal{E} = \mathcal{U} \cup \Inj \mathcal{A}$ yields $\mathcal{S}_{\mathcal{E}} =  \mathcal{S} \cap \mathcal{S}_{\mathcal{E}_\top}$. Now by Lemma \ref{sim} the functor $\coker \Hom_{A}(f,-)$ is simple in $\Fp(\mod A, \Ab)/\mathcal{S}_{\mathcal{E}}$. Since $f$ is a monomorphism, we have $\coker \Hom_{A}(f,-) \in \mathcal{S}_{\mathcal{E}_\top}$. It follows that $\coker \Hom_{A}(f,-)$ is simple in $\Fp(\mod A, \Ab)/\mathcal{S}$. By Proposition \ref{generic} there exists an endofinite module $M\in \Mod A$ with $\mathcal{S} = \mathcal{S}_M$. Hence, $\mathcal{E}_{M} = \mathcal{E}$ and the assignment is surjective.
\end{proof}

\begin{rem}\label{indenough}\rm Theorem \ref{genericex} shows the importance of almost $\mathcal{E}$-exact sequences. The following are fundamental properties about them.
\begin{enumerate}
    \item[(1)] Let $0 \rightarrow X \xrightarrow{f} Y \rightarrow Z \rightarrow 0$ be almost $\mathcal{E}$-exact and $X'$ a direct summand of $X$. Then the canonical morphisms $g\colon X/X' \rightarrow Y/f(X')$ and $h\colon X' \rightarrow Y$ yield a short exact sequence
    \begin{equation*}
         0 \rightarrow \coker \Hom_{A}(g,-) \rightarrow \coker  \Hom_{A}(f,-) \rightarrow \coker \Hom_{A}(h ,-) \rightarrow 0.
    \end{equation*}
    in $\Fp(\mod A, \Ab)$. By Lemma \ref{sim} either $g$ or $h$ is almost $\mathcal{E}$-monic. Hence, by iteratively splitting off direct summands from $X$ there exists an almost $\mathcal{E}$-exact sequence $0 \rightarrow X' \rightarrow Y' \rightarrow Z' \rightarrow 0$, where $X'$ is an indecomposable direct summand of $X$.
    \item[(2)] By the proof of Lemma \ref{sim} the condition for a morphism $X\to Y$ in $\mod A$, which is not an $\mathcal{E}$-monomorphism, to be almost $\mathcal{E}$-monic only needs to be tested for indecomposable $Z\in \mod A$. That is, for every morphism $X\rightarrow Z$, the morphism $\varphi$ in some (weak) pushout diagram
\begin{equation*}
    \begin{tikzcd}
        X \arrow[r] \arrow[d]& Y \arrow[d]\\
        Z \arrow[r, "\varphi"] & P
    \end{tikzcd}
\end{equation*}
is an $\mathcal{E}$-monomorphism or the induced morphism $X \rightarrow Y \oplus Z$ is an \linebreak{$\mathcal{E}$-monomorphism}.
    \item[(3)] One can define \emph{almost $\mathcal{E}$-epic} morphisms in $\mod A$ dually to almost $\mathcal{E}$-monic morphisms. Then for a short exact sequence $0 \rightarrow X \rightarrow Y \rightarrow  Z \rightarrow 0$
    the morphisms $X\rightarrow Y$ is almost $\mathcal{E}$-monic if and only if $Y\rightarrow Z$ is almost $\mathcal{E}$-epic. This follows from the duality in Lemma \ref{dual} and from Lemma \ref{sim} applied to both $\mod A$ and $(\mod A)^\textnormal{op}$.
\end{enumerate}
\end{rem}We continue by investigating exact structures $\mathcal{E}$ on $\mod A$ via ideals of morphisms in $\mod A$. A class of morphisms $\mathcal{I}$ in $\mod A$ is an \emph{ideal} if 
\begin{itemize}
    \item[(1)] for all $f, g \in \mathcal{I}$ the sum $f + g$ is in $\mathcal{I}$, and
    \item[(2)] for all $f \in \mathcal{I}$ and arbitrary $\alpha,\beta$ the composition $\beta f \alpha$ is in $\mathcal{I}$,
\end{itemize}
whenever the expressions are defined. For $X,Y\in \mod A$ let $\mathcal{I}(X,Y)$ denote the collection of all morphisms in $\mathcal{I}$ starting in $X$ and ending in $Y$. This induces functors $\mathcal{I}(X,-)\colon \mod A \rightarrow \Ab$ and $\mathcal{I}(-,Y)\colon \mod A \rightarrow \Ab^\op$. 

The ideal $\mathcal{I}$ is \emph{fp-idempotent} if 
\begin{align*}
    \mathcal{S}_\mathcal{I} = \{F\in \Fp(\mod A, \Ab) \mid F(f) = 0 \text{ for all }f\in \mathcal{I}\}
\end{align*}
is a Serre subcategory of $\Fp(\mod A, \Ab)$ \cite[Section 5.1]{Krause}. 
\begin{coro}\label{ideall}{\rm \cite[Corollary 5.9]{Krause}} There exists a one to one correspondence between Serre subcategories $\mathcal{S}$ of $\Fp(\mod A, \Ab)$ and fp-idempotent ideals $\mathcal{I}$ of $\mod A$, given by $\mathcal{I}\mapsto \mathcal{S}_\mathcal{I}$ and
\begin{align*}
    \mathcal{S} \mapsto \{f \colon X \rightarrow Y \text{ in }\mod A\mid F(f) = 0 \text{ for all }F \in \mathcal{S}\}.
\end{align*}
\end{coro} 

Given an exact structure $\mathcal{E}$ on $\mod A$ let $\mathcal{S}_\mathcal{E}$ be the associated Serre subcategory of $\Fp(\mod A, \Ab)$ by Corollary \ref{enosim} and  $\mathcal{I}_\mathcal{E}$ the corresponding fp-idempotent ideal by the above result. Following the correspondences, the ideal $\mathcal{I}_\mathcal{E}$ equals all $f\colon X \rightarrow Y$ such that for all $\mathcal{E}$-monomorphisms $X \rightarrow M$ there exists $M \rightarrow Y$ making the diagram
\begin{equation*}
    \begin{tikzcd}
        Y & \\
        X \arrow[u, "f"] \arrow[r] & M \arrow[ul] 
    \end{tikzcd}
\end{equation*}
commute. Motivated by this description, we call $\mathcal{I}_\mathcal{E}$ the \emph{$\mathcal{E}$-injectivity ideal}. Dually, we define the $\mathcal{E}$\emph{-projective ideal} and denote it by $\mathcal{P}_\mathcal{E}$. These ideals were already considered in \cite[Section 9.2]{Gabriel}.

\begin{rem}\label{du}\rm Let $D\colon \mod k \rightarrow (\mod k)^\textnormal{op}$ be the Matlis duality, which induces a duality between $\mod A$ and $\mod (A^\op)$. For an ideal $\mathcal{I}$ of $\mod A$ we denote by $D \mathcal{I}$ the ideal of all morphisms in $\mod (A^\op)$ isomorphic to $D \varphi$ for some $\varphi \in \mathcal{I}$.
\begin{itemize}
    \item[(1)] The assignment $F \mapsto DFD$ defines a duality between $\Fp(\mod A, \Ab)$ and $\Fp(\mod (A^\textnormal{op}), \Ab)$ by \cite[Proposition 3.3]{Auslander}. Further, $DFD(D f) = 0$ if and only if $F(f) = 0$ for a morphism $f$ in $\mod A$. It follows that an ideal $\mathcal{I}$ of $\mod A$ is fp-idempotent if and only if $D \mathcal{I}$ is fp-idempotent. 
    \item[(2)] For an exact structure $\mathcal{E}$ on $\mod A$ let $D \mathcal{E}$ denote all kernel-cokernel pairs in $\mod (A^\op)$ isomorphic to $(Df,Dg)$ for $(f,g) \in \mathcal{E}$. Then $D \mathcal{E}$ is an exact structure on $\mod (A^\op)$. Clearly $\mathcal{I}_{D\mathcal{E}} = D \mathcal{P}_\mathcal{E}$ and $\mathcal{P}_{D\mathcal{E}} = D \mathcal{I}_\mathcal{E}$. In particular, the $\mathcal{E}$-projectivity ideal is also always fp-idempotent by (1).
\end{itemize}
\end{rem}
With the above remark, Theorem \ref{sum} and Corollary \ref{ideall} imply the following.

\begin{coro}\label{exideal} There exists a one to one correspondence between the following.
\begin{itemize}
    \item[\rm (1)] Exact structures $\mathcal{E}$ on $\mod A$.
    \item[\rm (2)] Fp-idempotent ideals $\mathcal{I}$ of $\mod A$ containing all morphisms factoring through an injective module.
    \item[\rm (3)] Fp-idempotent ideals $\mathcal{P}$ of $\mod A$ containing all morphisms factoring through a projective module. 
\end{itemize}
The assignments are given by $\mathcal{E} \mapsto \mathcal{I}_\mathcal{E}$ and $\mathcal{E} \mapsto \mathcal{P}_\mathcal{E}$, as well as
\begin{align*}
    \mathcal{I} &\mapsto \mathcal{E}_\mathcal{I} = \{(f,\coker f) \mid \coker \Hom_A(f,-)(\varphi) = 0 \text{ for all }\varphi \in \mathcal{I} \},\\
    \mathcal{P} &\mapsto \mathcal{E}_{\mathcal{P}} = \{(\ker g, g) \mid \coker \Hom_A (-,g)(\psi) = 0 \text{ for all }\psi\in \mathcal{P}\}.
\end{align*}
\end{coro}

The following offers a different description of the $\mathcal{E}$-injectivity ideal $\mathcal{I}_\mathcal{E}$.

\begin{lem}\label{ideal} Let $\mathcal{E}$ be an exact structure on $\mod A$ and $f$ a morphism in $\mod A$.
\begin{itemize}
    \item[(1)] If $f$ factors through an fp-$\bar{\mathcal{E}}$-injective module in $\Mod A$, then $f\in \mathcal{I}_\mathcal{E}$.
    \item[(2)] If $f\in \mathcal{I}_\mathcal{E}$, then $f$ factors through a finite direct sum of indecomposable $\bar{\mathcal{E}}$-injective modules in $\Mod A$.
\end{itemize}
\end{lem}

\begin{proof} By Theorem \ref{sum} the definable subcataegory $\mathcal{X}_\mathcal{E}$ corresponding to $\mathcal{E}$ equals the collection all fp-$\bar{\mathcal{E}}$-injectives in $\Mod A$ and the corresponding closed set $\mathcal{U}_\mathcal{E}$ in $\Mod A$ coincides with the indecomposable $\bar{\mathcal{E}}$-injectives in $\Mod A$. Now (1) follows by \cite[Theorem 5.2]{Krause} and (2) by \cite[Corollary 4.7]{Prest05}.
\end{proof}

Let $\tau \colon \pmod A \rightarrow \imod A$ denote the Auslander-Reiten translation, where $\pmod A$ is the projectively stable module category and $\imod A$ the injectively stable module category. Further, let $\tau^-$ denote the inverse of $\tau$. The following result is a relative Auslander-Reiten formula for the lifted exact structure $\bar{\mathcal{E}}$ on $\Mod A$.

\begin{prop}\label{bigarf} Let $\mathcal{E}$ be an exact structure on $\mod A$ and $\mathcal{I}$ the collection of all morphisms factoring through an $\bar{\mathcal{E}}$-injective module. There exists an isomorphism
\begin{align*}
    D \,\Ext_{\bar{\mathcal{E}}}^1(C, X) \cong \Hom_{A}(X, \tau C)/\mathcal{I}(X, \tau C)
\end{align*}
functorial in $X\in \Mod A$ and $C\in \mod A$.
\end{prop}

\begin{proof} For the extension groups relative to exact structures, see Section 1.2. We proceed as in \cite{Krause3a}. By Proposition \ref{enough} there exists an $\bar{\mathcal{E}}$-injective envelope $X\rightarrow Q$. Let $K$ be its cokernel. For the short exact sequence $\delta\colon 0 \rightarrow X \rightarrow Q \rightarrow K \rightarrow 0$, the covariant defect $\delta_*$ and the contravariant defect $\delta^*$ are given by the exact sequences
\begin{align*}
    0 \rightarrow \Hom_{A}(K,-) \rightarrow \Hom_{A}(Q,-) \rightarrow \Hom_{A}(X,-) \rightarrow \delta_{*} \rightarrow 0,\\
    0 \rightarrow \Hom_{A}(-,X) \rightarrow \Hom_{A}(-,Q) \rightarrow \Hom_{A}(-,K) \rightarrow \delta^{*} \rightarrow 0.
\end{align*}
The defect formula states $D\,\delta^*(C) \cong \delta_*(\tau C)$ functorial in $\delta$ and $C$ \cite[Theorem]{Krause3a}. Now $\delta^*(C)  \cong \Ext^1_{\bar{\mathcal{E}}}(C,X)$ and $\delta_*(\tau C) = \Hom_{A}(X,\tau C)/\mathcal{I}(X,\tau C)$.
\end{proof}

As a consequence, we also get the relative Auslander-Reiten formulas for exact structures $\mathcal{E}$ on $\mod A$, which are shown in \cite[Corollary 9.4]{Gabriel}. They also hold in a more general setup, see \cite{Lenzing}, and their proof does not rely on any arguments involving big objects contrary to our approach.  

\begin{coro}\label{arf} Let $\mathcal{E}$ be an exact structure on $\mod A$. For $X,Y \in \mod A$ there exist functorial isomorphisms
\begin{align*}
    \Ext_{\mathcal{E}}^1(X,Y) &\cong D\,\Hom_{A}(Y,\tau X)/\mathcal{I}_\mathcal{E}(Y,\tau X),\\
    \Ext_{\mathcal{E}}^1(X,Y) &\cong D\,\Hom_{A}(\tau^{-} Y, X)/\mathcal{P}_\mathcal{E}(\tau^{-} Y, X).
\end{align*}
\end{coro}

\begin{proof} The first equality follows by Proposition \ref{bigarf} with Lemma \ref{ideal}. The second equality follows by duality. 
\end{proof}

For an ideal $\mathcal{I}$ of $\mod A$ let $\underline{\mathcal{I}}$ denote the induced ideal in $\underline{\textnormal{mod}}\,A$ and $\overline{\mathcal{I}}$ the induced ideal in $\overline{\textnormal{mod}}\,A$. Further, we denote by $\underline{\textnormal{Zsp}}\, A$, respectively $\overline{\textnormal{Zsp}}\, A$, all closed sets in $\Ind A$ containing $\textnormal{proj}\, A$, respectively $\Inj A$.

\begin{prop}\label{injproj} Let $\mathcal{E}$ be an exact structure on $\mod A$. The Auslander-Reiten translation $\tau$ and its inverse $\tau^{-}$ induce mutually inverse bijections
\begin{equation*}
\begin{tikzcd}
    \underline{\mathcal{P}}_{\mathcal{E}} \arrow[r, "\tau", shift left] & \overline{\mathcal{I}}_{\mathcal{E}} \arrow[l, "\tau^{-}", shift left].
\end{tikzcd}
\end{equation*}
In particular, $1_X\in \mathcal{P}_{\mathcal{E}}$ if and only if $1_{\tau X}\in \mathcal{I}_{\mathcal{E}}$ for $X\in \mod A$.
\end{prop}

\begin{proof} Let $\delta\colon 0 \rightarrow X\rightarrow Y \rightarrow Z \rightarrow 0$ be an $\mathcal{E}$-exact sequence. By definition of the $\mathcal{E}$-projectivity and $\mathcal{E}$-injectivity ideal, we have $f \in \mathcal{P}_{\mathcal{E}}$ iff $\delta^*(f) = 0$ and $f \in \mathcal{I}_{\mathcal{E}}$ iff $\delta_*(f) = 0$ for all $\delta$. The defect formula $D \delta^* \cong \delta_* \tau$ implies the desired result. 
\end{proof}

\begin{coro}\label{arfbig} There exist mutually inverse order preserving assignments  
\begin{equation*}
\begin{tikzcd}
\underline{\textnormal{Zsp}}\, A \arrow[r, "\tau", shift left] & \overline{\textnormal{Zsp}}\, A \arrow[l, "\tau^-", shift left]
\end{tikzcd}
\end{equation*}
such that for all indecomposable $X\in \mod A$ we have
\begin{itemize}
    \item[(1)] $X\in \mathcal{U}$ if and only if $\tau X \in \tau \,\mathcal{U}$ for $\mathcal{U} \in \underline{\textnormal{Zsp}}\,A$, and
    \item[(2)] $X\in \mathcal{U}$ if and only if $\tau^- X \in \tau^- \, \mathcal{U}$ for $\mathcal{U} \in \overline{\textnormal{Zsp}}\,A$.
\end{itemize}
\end{coro}
\begin{proof}

By \cite[Corollary 4.7]{Prest05} closed sets in $\Ind A$ containing $\textnormal{proj}\,A$, respectively $\Inj A$, correspond to fp-idempotent ideals containing all morphisms factoring through an injective, respectively projective, module in $\mod A$. The assignment is now given by Corollary \ref{exideal} and the desired properties follow by Proposition \ref{injproj}.
\end{proof}

In \cite[Corollary 5.15]{Krause0} the Auslander-Reiten translation has been extended to a homeomorphism $\Ind A \setminus \textnormal{proj}\, A \to \Ind A \setminus \Inj A$. This also induces a bijection $\underline{\textnormal{Zsp}} \,A \to \overline{{\textnormal{Zsp}}} \,A$ and it seems plausible that it is the same as in the above corollary.  

\begin{exa}\label{alm}\rm We consider the smallest exact structure $\mathcal{E} = \mathcal{E}_{\textnormal{fin}}$ containing all almost split sequences in $\mod A$. Exact structures containing $\mathcal{E}$ are relevant for the existence of generic modules by Theorem \ref{genericex}.  

The Serre subcategory  $\mathcal{S}_{\mathcal{E}}$ of $\Fp(\mod{A}, \Ab)$ corresponding to $\mathcal{E}$ is generated by all $\coker \Hom_{A}(f,-)$, where $f$ is a left almost split morphism starting in a non-injective indecomposable module, see Corollary \ref{enosim}. The corresponding definable subcategory $\mathcal{X}_{\mathcal{E}}$ consists of all $X\in \Mod{A}$ such that $\Hom_{A}(f,X)$ is surjective for all such $f$, see Section 1.6 (i)$\leftrightarrow$(iii). Thus, $\mathcal{X}_\mathcal{E}$ equals the collection of all modules in $\Mod A$ having no non-injective indecomposable direct summand of finite length. It follows that the indecomposable $\bar{\mathcal{E}}$-injective modules form the closed set $\mathcal{U} \cup \Inj A$ in $\Ind A$, where $\mathcal{U}$ denotes all infinite length modules in $\Ind A$, see Theorem \ref{sum}. 

By \cite[Corollary 8.13]{Krause} the fp-idempotent ideal $\mathcal{I}_{\mathcal{E}}$ corresponding to $\mathcal{X}_\mathcal{E}$ equals $\radA^\omega + \langle \textnormal{inj}\,A \rangle$. Here, $\langle \textnormal{inj}\,A \rangle$ denotes all morphisms in $\mod{A}$ factoring through an injective module, $\radA$ the radical ideal of $\mod{A}$ and $\radA^\omega = \bigcap_{n=1}^\infty \radA^n$. The ideal $\mathcal{I}_\mathcal{E}$ is the $\mathcal{E}$-injectivity ideal. Since the duality $D$ preserves almost split sequences and radical morphisms, it follows that the $\mathcal{E}$-projectivity ideal $\mathcal{P}_{\mathcal{E}}$ equals $\radA^\omega + \langle \textnormal{proj}\,A \rangle$. In particular, $\tau (\mathcal{U} \cup \textnormal{proj}\,A) = \mathcal{U} \cup \Inj A$, see Corollary \ref{arfbig}. Applying the relative Auslander-Reiten formulas in Corollary \ref{arf} yields
\begin{align*}
\Ext^1_{\mathcal{E}_{\textnormal{fin}}}(X,Y) &\cong D\,\Hom_{A}(Y,\tau X)/(\radA^\omega + \langle \textnormal{inj}\,A \rangle)(Y,\tau X),\\
\Ext^1_{\mathcal{E}_{\textnormal{fin}}}(X,Y) &\cong D\,\Hom_{A}(\tau^- Y,X)/(\radA^\omega + \langle \textnormal{proj}\,A \rangle)(\tau^- Y, X).
\end{align*}
We consider some special cases.

If $X$ has no preprojective (in the sense of \cite{Auslander0}) indecomposable direct summand, then $\Hom_{A}(A,X) \subseteq \radA^\omega(A,X)$ by \cite[Proposition 5.5]{self}. In this case 
\begin{align*}
\Ext^1_{\mathcal{E}_{\textnormal{fin}}}(X,Y) &\cong D\,\Hom_{A}(\tau^{-} Y, X)/\radA^\omega(\tau^-Y, X).
\end{align*}
Similarly, if $Y$ has no preinjective indecomposable direct summand, then 
\begin{align*}
\Ext^1_{\mathcal{E}_{\textnormal{fin}}}(X,Y) &\cong D\,\Hom_{A}(Y, \tau X)/\radA^\omega(Y, \tau X).
\end{align*}
It follows that $\Ext^1_{\mathcal{E}_{\textnormal{fin}}}(\tau X, Y) \cong \Ext^1_{\mathcal{E}_{\textnormal{fin}}}(X, \tau^- Y)$ if $\tau X$ has no preprojective and $\tau^- Y$ no preinjective direct summand.

We also consider the condition $ \radA^\omega( Y, \tau X) \subseteq \langle \textnormal{inj}\,A\rangle$. This holds if and only if 
\begin{align*}
    \Ext^1_{\mathcal{E}_{\textnormal{fin}}}(X,Y) &\cong D\,\Hom_{A}(Y,\tau X)/\langle \textnormal{inj}\,A\rangle(Y,\tau X) \cong \Ext^1_{A}(X,Y)
\end{align*}
by the classical Auslander-Reiten formula. Thus, we have a criterion when every short exact sequence between two modules is generated by almost split sequences.
\end{exa}

\section{Fp-idempotent ideals}

In this section we stick to the case of a module category over an Artin algebra $A$. By the results in Section 4, the importance of fp-idempotent ideals to study exact structures on $\mod A$ is apparent. Our aim will be to further analyze them. To do so, we need some additional terminology.

Let $\mathcal{I}$ be an ideal of $\mod A$. A morphism $f\colon X \rightarrow M$ in $\mod A$ is called (\emph{strongly}) \emph{left} $\mathcal{I}$\emph{-factoring} if for all $\varphi \colon X\rightarrow Y$ in $\mathcal{I}$ there exists $g\colon Y \rightarrow M$ (in $\mathcal{I}$) such that $f = g\varphi$. Note that $f$ is left $\mathcal{I}$-factoring if and only if $\coker \Hom_{A}(f,-) (\varphi) = 0$ for all $\varphi \in \mathcal{I}$. The ideal $\mathcal{I}$ is \emph{left idempotent} if every left $\mathcal{I}$-factoring morphism is strongly left $\mathcal{I}$-factoring. The definition of (\emph{strongly}) \emph{right $\mathcal{I}$-factoring} morphisms and \emph{right idempotent} ideals is dual. Clearly, every idempotent ideal $\mathcal{I}$ (so $\mathcal{I}^2 = \mathcal{I}$) is both left and right idempotent. The following is an important recharacterization of fp-idempotent ideals.

\begin{lem}\label{char} Let $\mathcal{I}$ be an ideal of $\mod A$. Then $\mathcal{I}$ is fp-idempotent if and only if it is left idempotent if and only if it is right idempotent.
\end{lem}

\begin{proof} Assume that $\mathcal{I}$ is fp-idempotent. Let $\mathcal{S}$ be the corresponding Serre subcategory of $\Fp(\mod A, \Ab)$ from Corollary \ref{ideall} and $f\colon X\rightarrow M$ left $\mathcal{I}$-factoring. Then $\coker \Hom_{A}(f,-)(\varphi) =0$ for $\varphi\in \mathcal{I}$. Thus, $\coker \Hom_{A}(f,-) \in \mathcal{S}$. Recall that there exists a definable subcategory $\mathcal{X}$ of $\Mod A$ corresponding to $\mathcal{S}$, given by 
\begin{align*}
    \mathcal{X} = \{Z\in \Mod A \mid F(Z) = 0\textnormal{ for all }F\in \mathcal{S}\},
\end{align*}
see Section 1.6 (i)$\leftrightarrow$(iii). By \cite[Theorem 5.2]{Krause} the ideal $\mathcal{I}$ equals all morphisms in $\mod A$ that factor through a module in $\mathcal{X}$. In particular, every $\varphi \colon X \rightarrow Y$ in $\mathcal{I}$ factors through some $Z\in \mathcal{X}$ and $\coker \Hom_A(f,Z) = 0$. This yields a commutative diagram
\begin{equation*}
    \begin{tikzcd}
        Y &\\
        Z \arrow[u] & \\
        X \arrow[u] \arrow[r] & M \arrow[ul] \arrow[uul]. 
    \end{tikzcd}
\end{equation*}
Hence, $M \rightarrow Y$ is in $\mathcal{I}$ and $X\rightarrow Y$ is strongly left $\mathcal{I}$-factoring . Thus, the ideal $\mathcal{I}$ is left idempotent.

Assume that $\mathcal{I}$ is left idempotent. For $\mathcal{I}$ to be fp-idempotent, the collection
\begin{align*}
    \mathcal{S} = \{F\in \Fp(\mod A, \Ab) \mid F(\varphi) = 0 \text{ for all }\varphi\in \mathcal{I}\}
\end{align*}
must be a Serre subcategory of $\Fp(\mod A, \Ab)$. Now $\mathcal{S}$ is always closed under subobjects and quotients. It remains to show that $\mathcal{S}$ is closed under extensions. A short exact sequence in $\Fp(\mod A, \Ab)$ is given by 
\begin{align*}
    0 \rightarrow \coker \Hom_{A}(f,-) \rightarrow \coker \Hom_{A}(g,-) \rightarrow \coker \Hom_{A}(h,-) \rightarrow 0,
\end{align*}
where $g \colon X \rightarrow M$ is arbitrary, $f \colon N \rightarrow P$ is given by a pushout diagram 
\begin{equation*}
    \begin{tikzcd}
        X \arrow[r, "g"] \arrow[d, "\alpha", swap] & M \arrow[d, "\beta"] \\
        N \arrow[r, "f"] & P 
    \end{tikzcd}
\end{equation*}
and $h \colon X \rightarrow M \oplus N$ is induced by $g, \alpha$. If the right term in the short exact sequence is in $\mathcal{S}$, then for all $\varphi \colon X \rightarrow Y$ in $\mathcal{I}$ there exists $\psi \colon M\oplus N \rightarrow Y$ such that $\varphi = \psi h = \psi_1 g + \psi_2 \alpha$ with $\psi_1 \colon M \rightarrow Y$ and $\psi_2 \colon N \rightarrow Y$. Because $\mathcal{I}$ is left idempotent, we can choose $\psi \in \mathcal{I}$. In particular, $\psi_2 \in \mathcal{I}$. Now if also the left term in the short exact sequence is in $\mathcal{S}$, then there exists $\gamma \colon P \rightarrow Y$ with $\psi_2 = \gamma f$. It follows that
\begin{align*}
    \varphi = \psi_1 g + \psi_2 \alpha = \psi_1 g + \gamma f \alpha = (\psi_1 + \gamma \beta ) g.
\end{align*}
Hence, also the middle term in the short exact sequence is in $\mathcal{S}$ and the ideal $\mathcal{I}$ is fp-idempotent.

Moreover, $\mathcal{I}$ is right idempotent if and only if $\mathcal{I}$ is fp-idempotent by duality and Remark \ref{du} (1). \end{proof}

We will apply the recharacterization of fp-idempotent ideals to generalize the following phenomenon. Let $\mathcal{C}$ be a full additive subcategory of $\mod A$ and $\mathcal{I}$ the ideal of all morphisms factoring through some module in $\mathcal{C}$. Every morphism $X \rightarrow Y$ in $\mathcal{I}$ factors as $X\rightarrow C \rightarrow Y$ with $C\in \mathcal{C}$ and taking images yields a second factorization
\begin{align*}
    X\longrightarrow C' \longrightarrow C'' \longrightarrow Y,
\end{align*}
where $X\rightarrow C'$ is an epimorphism ending in a submodule of $C$ and $C'' \rightarrow Y$ a monomorphism starting in a quotient of $C$. Moreover, $C' \rightarrow C''$ factors through $C$. Now submodules $C'$ and quotients $C''$ of modules in $\mathcal{C}$ can also be described as all modules that fulfill 
\begin{align*}
\mathcal{I}(C', DA) = \Hom_{A}(C', DA) \quad \textnormal{and} \quad \mathcal{I}(A, C'') = \Hom_{A}(A, C'')     
\end{align*}
respectively. The above generalizes to fp-idempotent ideals.

\begin{prop}\label{sten} Let $\mathcal{I}$ be an fp-idempotent ideal of $\mod A$. Then every morphism $X\rightarrow Y$ in $\mathcal{I}$ factors as 
\begin{align*}
    X \longrightarrow C' \longrightarrow C'' \longrightarrow Y, 
\end{align*}
where $X\rightarrow C'$ is an epimorphism, $C'' \rightarrow Y$ is a monomorphism and $C' \rightarrow C''$ is in $\mathcal{I}$ with
\begin{align*}
    \mathcal{I}(C', DA) = \Hom_{A}(C', DA) \quad \text{and} \quad \mathcal{I}(A, C'') = \Hom_{A}(A, C'').
\end{align*}
\end{prop}
\begin{proof} For $C \in \mod A$ let $tC$ be the sum of all images of morphisms in $\mathcal{I}$ ending in $C$. Because $C$ is of finite length, this image is obtained from a single epimorphism $\varphi \colon M \rightarrow tC$ in $\mathcal{I}$. If $tC = C$, then $\mathcal{I}(A,C) = \Hom_A(A,C)$, since every morphism $A \rightarrow C$ in $\mathcal{I}$ factors through $\varphi$. By construction the inclusion $tC \rightarrow C$ is right $\mathcal{I}$-factoring. By Lemma \ref{char} it follows that every morphism $X\rightarrow Y$ in $\mathcal{I}$ factors as
\begin{align*}
    X\longrightarrow tY \longrightarrow Y
\end{align*}
with $X\rightarrow tY$ in $\mathcal{\mathcal{I}}$. Again $X\rightarrow tY$ factors as $X\rightarrow t(tY) \rightarrow tY$. Because $tY$ equals the image of all morphisms in $\mathcal{I}$ ending in $Y$, we must have $t(tY) = tY$. Hence, we found the desired $C'' = tY$. The existence of $C'$ and the desired factorization follows dually.
\end{proof}

For an fp-idempotent ideal $\mathcal{I}$ we denote by $s(\mathcal{I})$, respectively $e(\mathcal{I})$, the full subcategory of all $C\in \mod A$ with $\mathcal{I}(C, DA) = \Hom_{A}(C, DA)$ and respectively $\mathcal{I}(A, C) = \Hom_{A}(A, C)$. Motivated by Proposition \ref{sten} we say that $\mathcal{I}$ \emph{starts in} $s(\mathcal{I})$ and \emph{ends in} $e(\mathcal{I})$.

\begin{exa}\rm \begin{itemize}
    \item[(1)] Let $\mathcal{C}$ be a full additive subcategory of $\mod A$ and $\mathcal{I}$ the ideal of morphisms factoring through a module in $\mathcal{C}$. Then $s(\mathcal{I})$, respectively $e(\mathcal{I})$, equals all submodules, respectively quotients, of modules in $\mathcal{C}$.
    \item[(2)] Consider the fp-idempotent ideal $\mathcal{I} = \radA^\omega$. Then $s(\mathcal{I})$, respectively $e(\mathcal{I})$, equals all modules in $\mod A$ with no preinjective, respectively preprojective, direct summand (see Example \ref{alm}).
\end{itemize}
\end{exa}

\begin{rem}\label{closure}\rm The following describes some closure properties of the collection of fp-idempotent ideals. Let $\{\mathcal{I}_i\}_{i\in I}$ be a set of fp-idempotent ideals of $\mod A$ and $\mathcal{S}_i$ the corresponding Serre subcategories of $\Fp(\mod A, \Ab)$ from Corollary \ref{ideall}. 
\begin{itemize}
    \item[(1)] The ideal $\mathcal{J} = \sum_{i\in I} \mathcal{I}_i$ contains all finite sums of morphisms that are each contained in some $\mathcal{I}_i$. It is easy to check that $\mathcal{J}$ is the fp-idempotent ideal corresponding to the Serre subcategory $\bigcap_{i\in I} \mathcal{S}_i$. Hence, fp-idempotent ideals are closed under arbitrary sums.
    \item[(2)] Let $\mathcal{J} = \bigcap_{i\in I} \mathcal{I}_i$. Assume that the intersection is directed, that is for all $i,j \in I$ there is $l \in I$ with $\mathcal{I}_l \subseteq \mathcal{I}_i \cap \mathcal{I}_j$. It is easy to check that $\mathcal{J}$ is the fp-idempotent ideal corresponding to the Serre subcategory $\bigcup_{i\in I} \mathcal{S}_{i}$. Hence, fp-idempotent ideals are closed under directed intersections.
\end{itemize}
\end{rem}

Our next goal will be a new classification of fp-idempotent ideals. It is motivated by the following result, which describes some properties of the largest fp-idempotent ideal $\mathcal{J}$ contained in an arbitrary ideal $\mathcal{I}$. Note that $\mathcal{J}$ exists by Remark \ref{closure} (1).  

\begin{coro}\label{start}\cite[Corollary 5.11]{Krause} Let $\mathcal{I}$ be an ideal of $\mod A$. Then the set of fp-idempotent
ideals contained in $\mathcal{I}$ has a unique maximal element $\mathcal{J}$. It satisfies 
\begin{align*}
\bigcap_{n\in \mathbb{N}} \mathcal{I}^n \subseteq \mathcal{J} \subseteq \mathcal{I}. 
\end{align*}
\end{coro}

Let us denote $\mathcal{I}^\omega = \bigcap_{n\in \mathbb{N}} \mathcal{I}^n$. We show that for a class of ideals $\mathcal{I}$ the smallest fp-idempotent ideal contained in $\mathcal{I}$ equals $\mathcal{I}^\omega$. We need the following notions. 

Let $\mathcal{I}$ be an ideal of $\mod A$. A left $\mathcal{I}$-factoring morphism $f\colon X \rightarrow Y$ in $\mod A$ is a \emph{left $\mathcal{I}$-approximation} of $X$ if $f\in \mathcal{I}$. If every module $X\in \mod A$ admits a left $\mathcal{I}$-approximation, then $\mathcal{I}$ is \emph{covariantly finite}. The definition of a \emph{right} \mbox{\emph{$\mathcal{I}$-approximation}} and $\mathcal{I}$ being \emph{contravariantly finite} is dual. 

\begin{lem}\label{cov} Let $\mathcal{I}$ be a co- or contravariantly finite ideal of $\mod A$. Then $\mathcal{I}^\omega$ is the smallest fp-idempotent ideal contained in $\mathcal{I}$.
\end{lem}
\begin{proof} Let $\mathcal{J}$ be the maximal fp-idempotent ideal contained in $\mathcal{I}$. Then $\mathcal{I}^\omega \subseteq \mathcal{J}$ by Corollary \ref{start}. Assume that $\mathcal{I}$ is covariantly finite, let $f\colon X \rightarrow Y$ be in $\mathcal{J}$ and $X_i \rightarrow X_{i+1}$ a left $\mathcal{I}$-approximation for $i\geq 0$ with $X_0 = X$. Then $X_i \rightarrow X_{i+1}$ is left $\mathcal{J}$-factoring and hence strongly left $\mathcal{J}$-factoring by Lemma \ref{char}. Inductively, it follows that $X\rightarrow Y$ factors as 
\begin{equation*}
    \begin{tikzcd}
        Y \\
        \\
        X \arrow[r] \arrow[uu, "f"] & X_1 \arrow[r] \arrow[uul,"f_1"] & X_2 \arrow[r] \arrow[uull, "f_2", shift right = 0.2] & \dots \arrow[uulll, shift right = 1]
    \end{tikzcd}
\end{equation*}
with $f_i \in \mathcal{J}$. Since $X_i \rightarrow X_{i+1}$ is in $\mathcal{I}$ for all $i$, we conclude that $X\rightarrow Y$ is in $\mathcal{I}^\omega$ and so $\mathcal{J} = \mathcal{I}^\omega$. The argument is dual if $\mathcal{I}$ is contravariantly finite.
\end{proof}

\begin{thm}\label{new} An ideal $\mathcal{I}$ of $\mod A$ is fp-idempotent if and only if there exists a collection of ideals $\{\mathcal{J}_i\}_{i\in I}$ such that 
\begin{align*}
    \mathcal{I} = \bigcap_{i\in I} \mathcal{J}_i^\omega,
\end{align*}
where the intersection is directed. Moreover, we may choose a collection $\{\mathcal{J}_i\}_{i\in I}$ of covariantly finite ideals or of contravariantly finite ideals. 
\end{thm}

\begin{proof} Let $\mathcal{J}$ be an ideal of $\mod A$ and consider its finite powers $\mathcal{J}^n$ with $n \in \mathbb{N}$. Clearly, $(\mathcal{J}^n)^\omega = \mathcal{J}^\omega$. Hence, by Corollary \ref{start} the largest fp-idempotent ideal $\mathcal{J}_n$ contained in $\mathcal{J}^n$ contains $\mathcal{J}^\omega$. Thus, $\bigcap_{n\in \mathbb{N}} \mathcal{J}_n = \mathcal{J}^\omega$. This intersection is directed, since $\mathcal{J}^{n+1} \subseteq \mathcal{J}^n$ implies $\mathcal{J}_{n+1} \subseteq \mathcal{J}_n$. By Remark \ref{closure} (2) it follows that $\mathcal{J}^\omega$ is fp-idempotent as well as an arbitrary directed intersection of such ideals.     

Let $\mathcal{I}$ be an fp-idempotent ideal and $\mathcal{S}$ the corresponding Serre subcategory of $\Fp(\mod A, \Ab)$ from Corollary \ref{ideall}. For $F\in \mathcal{S}$ let $\mathcal{I}_F$ be the ideal of all morphisms $\varphi$ in $\mod A$ with $F(\varphi) = 0$. Then $\mathcal{I} \subseteq \mathcal{I}_F$ and $\mathcal{I} = \bigcap_{F\in \mathcal{S}} \mathcal{I}_F$. Moreover, the intersection is directed, since $\mathcal{I}_{F\oplus G} = \mathcal{I}_F \cap \mathcal{I}_G$ for $F,G\in \mathcal{S}$. We will show that the ideals $\mathcal{I}_F$ are covariantly finite. Then Lemma \ref{cov} implies $\mathcal{I} \subseteq \mathcal{I}_F^\omega$ for all $F\in \mathcal{S}$ and so $\mathcal{I} = \bigcap_{F\in \mathcal{S}} \mathcal{I}_F^\omega$.

For $F\in \mathcal{S}$ there exists $f\colon X\rightarrow M$ in $\mod A$ with $F\cong  \coker \Hom_{A}(f,-)$. Now $F(\varphi) = 0$ for a morphism $\varphi \colon Y\rightarrow N$ in $\mod A$ is equivalent to: For every $X\rightarrow Y$ there exists $M\rightarrow N$ such that the diagram
\begin{equation*}
\begin{tikzcd}
    Y \arrow[r, "\varphi"] & N\\
    X \arrow[u] \arrow[r, "f"]& M \arrow[u] 
\end{tikzcd}
\end{equation*}
commutes. Let $\Hom_{A}(X,Y)$ be generated by $g_1 , \dots , g_n$ as a $k$-module. This induces a morphism $g \colon X^n \rightarrow Y$. Consider the pushout diagram
\begin{equation*}
  \begin{tikzcd}
    Y \arrow[r, "\psi"] & P\\
    X^n \arrow[u, "g"] \arrow[r, "f^n"]& M^n. \arrow[u]
\end{tikzcd}  
\end{equation*}
We show that $\psi $ is a left $\mathcal{I}_F$-approximation. First, for $\varphi \colon Y \rightarrow N$ in $\mathcal{I}_F$ there exists a commutative diagram
\begin{equation*}
\begin{tikzcd}
    Y \arrow[r, "\varphi"] & N\\
    X^n \arrow[u, "g"] \arrow[r, "f"]& M^n. \arrow[u] 
\end{tikzcd}
\end{equation*}
It follows that $\varphi$ factors through $\psi$. Further, by construction of $g$ for all $X \rightarrow Y$ there exists $\lambda = (\lambda_1, \dots, \lambda_n)\in k^n$ such that the diagram 
\begin{equation*}
  \begin{tikzcd}
    Y \arrow[r, "\psi"] & P\\
    X^n \arrow[u, "g"] \arrow[r, "f^n"]& M^n \arrow[u]\\
    X\arrow[u, "\lambda"] \arrow[r, "f"] & M \arrow[u, "\lambda", swap]
\end{tikzcd}  
\end{equation*}
commutes. Hence $\psi \in \mathcal{I}_F$ and $\psi$ is a left $\mathcal{I}_F$-approximation.

We have found the desired collection $\{\mathcal{J}_i\}_{i\in I}$ as the collection $\{\mathcal{I}_F\}_{F\in \mathcal{S}}$. The ideals $\mathcal{I}_F$ are covariantly finite and by duality, see Remark \ref{du} (1), we can also choose $\mathcal{J}_i$ to be contravariantly finite. 
\end{proof}

\begin{exa}\label{useful}\rm 
\begin{itemize}
    \item[(1)] For an ideal $\mathcal{I}$ of $\mod A$ the ideal $\mathcal{I}^\omega$ is fp-idempotent by Theorem \ref{new}. A particular useful choice of $\mathcal{I}$ seems to be the following. For a full additive subcategory $\mathcal{C}$ of $\mod A$ let $\textnormal{rad}_\mathcal{C}$ be the collection of all morphisms $X \rightarrow Y$ that factor as $X \rightarrow C' \rightarrow C'' \rightarrow Y$, where $C',C'' \in \mathcal{C}$ and $C'\rightarrow C''$ is a radical morphism. Then $\textnormal{rad}_\mathcal{C}^\omega$ equals all morphisms $X\rightarrow Y$ that factor as
    \begin{align*}
        X\longrightarrow C_1 \longrightarrow \dots \longrightarrow C_n \longrightarrow Y
    \end{align*}
    for all $n\in \mathbb{N}$ with $C_i \in \mathcal{C}$ such that $C_i \rightarrow C_{i+1}$ is a radical morphism.
    \item[(2)] For an ideal $\mathcal{I}$ of $\mod A$ we define the expression $\mathcal{I}^\alpha$ for ordinal numbers $\alpha$ as in \cite{Prest98}. If $\lambda$ is a non-zero limit ordinal, let $\mathcal{I}^\lambda = \bigcap_{\alpha < \lambda} \mathcal{I}^\alpha$. If $\alpha$ is an arbitrary infinite ordinal, then $\alpha = \lambda + n$ for a limit ordinal $\lambda$ and $n\in \mathbb{N}$, and we let $\mathcal{I}^\alpha = (\mathcal{I}^\lambda)^{n+1}$. Since every non-zero limit ordinal $\lambda$ can be written as $\omega(\lambda'+ n)$ for $n\in \mathbb{N}$, a limit ordinal $\lambda'$ and $\omega$ the first non-finite ordinal, it follows that
    \begin{align*}
        \mathcal{I}^\lambda = \begin{cases} \left(\mathcal{I}^{\omega(\lambda'+n-1)}\right)^\omega & \textnormal{for } n >0,\\
        \bigcap_{\alpha<\lambda'}{\mathcal{I}^{\omega \alpha}} & \textnormal{for } n = 0.
        \end{cases}
    \end{align*}
In both cases the ideal $\mathcal{I}^\lambda$ is fp-idempotent by Theorem \ref{new} using induction in the second case. An important example is $\radA^\lambda$ for a limit ordinal $\lambda$.
\end{itemize}
\end{exa}

Let us connect some of the results in this section to exact structures on $\mod A$. Let $\mathcal{I}$ be an fp-idempotent ideal of $\mod A$ containing the ideal of all morphisms factoring through an injective module, denoted by $\langle \Inj A \rangle$. Then the corresponding exact structure $\mathcal{E}$ on $\mod A$ is given by all kernel-cokernel pairs $(f,g)$ such that $f$ is left $\mathcal{I}$-factoring, see Corollary \ref{exideal}. We may also start with an arbitrary \mbox{fp-idempotent} ideal $\mathcal{I}$ of $\mod A$ and consider all left $\mathcal{I}$-factoring monomorphisms. They will coincide with all left $(\mathcal{I}+\langle \Inj A \rangle)$-factoring morphisms and hence induce an exact structure $\mathcal{E}_{\mathcal{I}\textnormal{-inj}}$. Dually, also all right $\mathcal{I}$-factoring epimorphisms induce an exact structure $\mathcal{E}_{\mathcal{I}\textnormal{-proj}}$. One can also directly confirm that $\mathcal{E}_{\mathcal{I}\textnormal{-inj}}$ and $\mathcal{E}_{\mathcal{I}\textnormal{-proj}}$ are exact structures by checking the axioms, see Section 1.1, and using the fact that $\mathcal{I}$ is left and right idempotent, see Lemma \ref{char}.

The following generalizes the behavior of the exact structure $\mathcal{E}_\textnormal{fin}$ on $\mod A$ in Example \ref{alm}.

\begin{prop} Let $\mathcal{I}$ be an fp-idempotent ideal of $\mod A$ that starts in $\mathcal{C}$ and ends in $\mathcal{D}$.
\begin{itemize}
    \item[\rm (1)] There are isomorphisms 
    \begin{align*}
\Ext_{\mathcal{E}_{\mathcal{I}\textnormal{-inj}}}^1 (X,Y) &\cong D\,\Hom_{A}(Y,\tau X)/\mathcal{I}(Y,\tau X),\\
\Ext_{\mathcal{E}_{\mathcal{I}\textnormal{-proj}}}^1(X,Y) &\cong D\,\Hom_{A}(\tau^{-} Y, X)/\mathcal{I}(\tau^{-} Y, X),
    \end{align*}
    functorial in $X \in \mod A$ and $Y \in \mathcal{C}$ for the first isomorphism, and functorial in  $X\in \mathcal{D}$ and $Y \in \mod A$ for the second isomorphism. 
    \item[\rm (2)] If $\mathcal{\tau} \underline{\mathcal{I}} = \overline{\mathcal{I}}$, then $\mathcal{E} = \mathcal{E}_{\mathcal{I}\textnormal{-inj}} = \mathcal{E}_{\mathcal{I}\textnormal{-proj}}$ and
    \begin{align*}    \Ext^1_{\mathcal{E}}( \tau X, Y) \cong \Ext^1_{\mathcal{E}}(X,  \tau^{-} Y)
    \end{align*}
    for all $\tau X \in \mathcal{D}$ and $\tau^- Y\in \mathcal{C}$.
\end{itemize}
\end{prop}

\begin{proof} (1) The $\mathcal{E}_{\mathcal{I}\textnormal{-inj}}$-injective ideal equals $\mathcal{I}+ \langle \Inj A\rangle$. Moreover, if $Y\in \mathcal{C}$, then $\langle \Inj A \rangle (Y,-) \subseteq \mathcal{I}(Y,-)$, see Proposition \ref{sten}. Hence, the first isomorphism follows by Corollary \ref{arf}. The second isomorphism follows similarly. 

(2) If $\mathcal{\tau} \underline{\mathcal{I}} = \overline{\mathcal{I}}$ then $\mathcal{E} = \mathcal{E}_{\mathcal{I}\textnormal{-inj}} = \mathcal{E}_{\mathcal{I}\textnormal{-proj}}$ by Proposition \ref{injproj}. The isomorphism now follows from (1).
\end{proof}

The following characterizes almost $\mathcal{E}$-monic morphisms.

\begin{prop} Let $\mathcal{E}$ be an exact structure on $\mod A$ and $\mathcal{I}$ the $\mathcal{E}$-injectivity ideal. A morphism $f\colon X \rightarrow Y$ in $\mod A$ is almost $\mathcal{E}$-monic if and only if $\im \mathcal{I}(f,-)$ is a maximal element in the poset 
\begin{align*}
    \{\im \mathcal{I}(g,-) \neq \mathcal{I}(X,-) \mid g\colon X \rightarrow Z \text{ with }Z\in \mod A\}
\end{align*}
 of proper subfunctors of $\mathcal{I}(X,-)$.
\end{prop}

\begin{proof} A morphism $\alpha\colon M \rightarrow N$ is an $\mathcal{E}$-monomorphism if and only if it is strongly left $\mathcal{I}$-factoring by Lemma \ref{char}. Clearly, this is equivalent to $\im \mathcal{I}(\alpha,-) = \mathcal{I}(M,-)$.

The morphism $f$ is almost $\mathcal{E}$-monic if it is not an {$\mathcal{E}$-monomorphism} and for every morphism $g\colon X\rightarrow Z$, the morphism $\varphi$ in some pushout diagram
\begin{equation*}
    \begin{tikzcd}
        X \arrow[r, "f"] \arrow[d, "g", swap]& Y \arrow[d]\\
        Z \arrow[r, "\varphi"] & P
    \end{tikzcd}
\end{equation*}
is an $\mathcal{E}$-monomorphism or the morphism $h\colon X \rightarrow Y \oplus Z$ induced by $f$ and $g$ is an {$\mathcal{E}$-monomorphism}. These conditions are equivalent to $\im \mathcal{I}(f,-) \neq \mathcal{I}(X,-)$, and $\im \mathcal{I}(\varphi,-) = \mathcal{I}(Z,-)$ or $\im \mathcal{I}(g,-) = \mathcal{I}(X,-)$. By properties of the pushout and the direct sum, they are also equivalent to $\im \mathcal{I}(f,-) \neq \mathcal{I}(X,-)$, and
\begin{align*}
\im \mathcal{I}(g,-) \subseteq \im \mathcal{I}(f,-) \quad \textnormal{ or }\quad \im \mathcal{I}(f,-) + \im \mathcal{I}(g,-) = \mathcal{I}(X,-).
\end{align*}
It follows that $\im \mathcal{I}(f,-)$ is a maximal element in the poset.
\end{proof}

\begin{exa}\label{kro}\rm Let $k$ be an algebraically closed field, $Q = \begin{tikzcd}
    1 \arrow[r, shift left] \arrow[r, shift right] & 2
\end{tikzcd}$ the Kronecker quiver and $A = kQ$. The indecomposable modules in $\mod A$ can be divided into three parts: The preprojective modules $\mathcal{P} = \{P_1, P_2, \dots \}$, the preinjective modules $\mathcal{Q} = \{Q_1, Q_2, \dots \}$ and the regular modules  $\mathcal{R}$, which further divide into tubes $\mathcal{R}^\lambda = \{R^\lambda_1, R^\lambda_2, \dots\}$ with $\lambda\in k\cup \{\infty \}$. The Auslander-Reiten quiver of $\mod A$ can be visualized as follows.
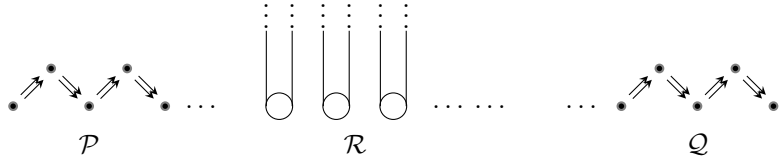
\tikzstyle{place}=[circle,draw=black!50,fill=black!100,thick,
inner sep=0pt,minimum size=1mm]
\begin{align*}
\begin{tikzpicture}
\draw (1, -0.5) node{$\mathcal{P}$};
\draw (9, -0.5) node{$\mathcal{Q}$};
\draw (4.5, -0.5) node{$\mathcal{R}$};
\draw (0,0) node[place]{};
\draw [-stealth](0.1,0.15) -- (0.35,0.4);
\draw [-stealth](0.15,0.1) -- (0.4,0.35);
\draw (0.5,0.5) node[place]{};
\draw [-stealth](0.65,0.4) -- (0.9,0.15);
\draw [-stealth](0.6,0.35) -- (0.85,0.1);
\draw (1,0) node[place]{};
\draw [-stealth](1.1,0.15) -- (1.35,0.4);
\draw [-stealth](1.15,0.1) -- (1.4,0.35);
\draw (1.5,0.5) node[place]{};
\draw [-stealth](1.65,0.4) -- (1.9,0.15);
\draw [-stealth](1.6,0.35) -- (1.85,0.1);
\draw (2,0) node[place]{};
\draw (2.5,0) node{$\dots$};
\draw (10,0) node[place]{};
\draw [-stealth](9.1,0.15) -- (9.35,0.4);
\draw [-stealth](9.15,0.1) -- (9.4,0.35);
\draw (9.5,0.5) node[place]{};
\draw [-stealth](9.65,0.4) -- (9.9,0.15);
\draw [-stealth](9.6,0.35) -- (9.85,0.1);
\draw (9,0) node[place]{};
\draw [-stealth](8.65,0.4) -- (8.9,0.15);
\draw [-stealth](8.6,0.35) -- (8.85,0.1);
\draw (8.5,0.5) node[place]{};
\draw [-stealth](8.1,0.15) -- (8.35,0.4);
\draw [-stealth](8.15,0.1) -- (8.4,0.35);
\draw (8,0) node[place]{};
\draw (7.5,0) node{$\dots$};
\draw (3.5,0) circle (5pt);
\draw (3.33,0)--(3.33,1);
\draw (3.33,1.3) node{$\vdots$};
\draw (3.67,0)--(3.67,1);
\draw (3.67,1.3) node{$\vdots$};
\draw (4.25,0) circle (5pt);
\draw (4.08,0)--(4.08,1);
\draw (4.08,1.3) node{$\vdots$};
\draw (4.42,0)--(4.42,1);
\draw (4.42,1.3) node{$\vdots$};
\draw (5,0) circle (5pt);
\draw (4.83,0)--(4.83,1);
\draw (4.83,1.3) node{$\vdots$};
\draw (5.17,0)--(5.17,1);
\draw (5.17,1.3) node{$\vdots$};
\draw (5.75,0) node{$\dots$};
\draw (6.3,0) node{$\dots$};
\end{tikzpicture}
\end{align*}
\end{exa}
Sometimes we identify the above collections with their additive closures. There only exist non-zero morphisms in $\radA^\omega$ from left to right in the above picture, that is from $\mathcal{P}$ to $\mathcal{R}$, from $\mathcal{R}$ to $\mathcal{Q}$ and from $\mathcal{P}$ to $\mathcal{Q}$. In the third case, such morphisms are always contained in $\radA^{\omega +1} = (\radA^\omega)^2$.

For $\lambda \in k \cup\{\infty\}$ the Pr\"ufer module $R_\infty^\lambda$ is constructed by a filtered colimit of monomorphisms in $\mathcal{R}^\lambda$ and the adic module $\hat{R}^\lambda$ by an inverse limit of epimorphisms in $\mathcal{R}^\lambda$. The Ziegler spectrum $\Ind A$ consists of all indecomposable modules in $\mod A$, the generic module $G$, the Pr\"ufer modules $R_\infty^\lambda$ and the adic modules $\hat{R}^\lambda$, see for example \cite[Theorem 14.2.15]{Krause4}. We will mainly focus on the non-empty closed sets $\mathcal{U} \subseteq \Ind A$ that contain no finite dimensional modules. They are given by two subsets $S,T \subseteq k \cup \{\infty \}$ such that
\begin{align*}
    \mathcal{U} = \{R_\infty^\lambda, \hat{R}^\mu, G\mid \lambda \in S, \mu \in T\}.
\end{align*} 
The corresponding fp-idempotent ideals are those contained in $\radA^\omega$ and the corresponding exact structures are those containing $\mathcal{E}_{\textnormal{fin}}$, see Example \ref{alm}. We make use of Theorem \ref{new}, in particular Example \ref{useful} (1), to produce these fp-idempotent ideals. 

\begin{itemize}
    \item[(1)] The fp-idempotent ideal $\textnormal{rad}_\mathcal{P}^\omega$ starts in $ \mathcal{P}$ and ends in $ \mathcal{R}\cup \mathcal{Q}$. It corresponds to the closed set $\mathcal{U} = \{\hat{R}^\lambda, G \mid \lambda \in k \cup \{\infty \}\}$. 
    \item[(2)] The fp-idempotent ideal $\textnormal{rad}_\mathcal{Q}^\omega$ starts in $ \mathcal{P} \cup \mathcal{R}$ and ends in $ \mathcal{Q}$. It corresponds to the closed set $\mathcal{U} = \{R_\infty^\lambda, G\mid \lambda \in k \cup\{\infty\}\}$.
    \item[(3)] For a non-empty subset $S\subseteq k\cup \{\infty\}$ let $\mathcal{R}^S = \bigcup_{\lambda\in S} \mathcal{R}^\lambda$. The fp-idempotent ideal $\textnormal{rad}_{\mathcal{R}^S}^\omega$ starts in $\mathcal{P}\cup \mathcal{R}^S$ and ends in $\mathcal{R}^S \cup \mathcal{Q}$. It corresponds to the closed set $\mathcal{U} = \{R_\infty^\lambda, \hat{R}^\lambda, G \mid \lambda \in S\}$.
    \item[(4)] We have $\radA^{\omega+1} = \bigcap \textnormal{rad}_{\mathcal{R}^S}^\omega$, where the intersection goes over all cofinite subsets $S \subseteq k \cup \{\infty\}$. Since the intersection is directed, the ideal $\radA^{\omega + 1}$ is fp-idempotent. It starts in $\mathcal{P}$ and ends in $\mathcal{Q}$. The corresponding closed subset equals $\mathcal{U} = \{G\}$. We can also describe $\radA^{\omega + 1}$ as a single $\omega$-power of an ideal as follows. For a sequence $\lambda_1, \lambda_2, \dots$ of pairwise distinct elements in $k \cup \{\infty\}$ let $\mathcal{C} = \{R^{\lambda_i}_j \mid i \in \mathbb{N}, 1\leq j \leq i\}$. Then $\radA^{\omega+1} = \textnormal{rad}_\mathcal{C}^\omega$.    
    \item[(5)] For $S,T \subseteq k \cup \{\infty \}$ with $S\neq \emptyset$ or $T\neq \emptyset$ let $\mathcal{I}_{S,T}$ be the ideal generated by proper monomorphisms in $\mathcal{R}^S$ and proper epimorphisms in $\mathcal{R}^T$. The \mbox{fp-idempotent} ideal $\mathcal{I}_{S,T}^\omega$ starts in $\mathcal{P}\cup \mathcal{R}^S$ and ends in $\mathcal{R}^T \cup \mathcal{Q}$. It corresponds to the closed set $\mathcal{U} = \{R_\infty^\lambda, \hat{R}^\mu, G \mid \lambda \in S, \mu \in T\}.$
\end{itemize}

Note that the ideal $\mathcal{I}_{S,T}$ in (5) fulfills $\tau \underline{\mathcal{I}}_{S,T} = \overline{\mathcal{I}}_{S,T}$. Thus, $\tau \underline{\mathcal{I}}_{S,T}^\omega = \overline{\mathcal{I}}_{S,T}^\omega$ and so $\tau(\mathcal{U} \cup \textnormal{proj}\,A) = \mathcal{U}\cup \Inj A$, see Corollary \ref{arfbig}. Because (1) - (3) are special cases of (5), and (4) is constructed from (3), it follows that $\tau (\mathcal{U} \cup \textnormal{proj}\, A) = \mathcal{U} \cup \Inj A$ in every case. Further, every fp-idempotent ideal here is an $\omega$-power of an ideal. It would be interesting to know for what algebras these properties hold. 

We continue by investigating the exact structure on $\mod A$ corresponding to the fp-idempotent ideal in (4). Let $\mathcal{E}$ be the exact structure such that the $\mathcal{E}$-injectivity ideal equals $\textnormal{rad}_A^{\omega+1}+\langle \Inj A \rangle$. By the observations before, the $\mathcal{E}$-projectivity ideal equals $\textnormal{rad}_{A}^{\omega+1} + \langle \textnormal{proj} \,A \rangle$. We consider an exact sequence
\begin{align*}
0 \longrightarrow X\xlongrightarrow{f} Y \xlongrightarrow{g} Z \longrightarrow 0   \tag{$\ast$}
\end{align*}
with $X, Z\in \mod A$ indecomposable. It is $\mathcal{E}$-exact if and only if every $X \rightarrow M$ in $\textnormal{rad}_{A}^{\omega+1}$ factors through $f$. If $X\notin \mathcal{P}$, then $\textnormal{rad}_{A}^{\omega+1}(X,-) = 0$ implies that ($\ast$) is $\mathcal{E}$-exact. Similarly, if $Z\notin Q$, then $(\ast)$ is $\mathcal{E}$-exact. It is left to consider the case $X \in \mathcal{P}$ and $Z\in \mathcal{Q}$. Clearly, if $(\ast)$ splits, then the sequence is $\mathcal{E}$-exact. If it does not split, then the identity $1_Z \colon Z \rightarrow Z$ does not factor through $g$. By the defect formula there also exists $X\rightarrow \tau Z$ that does not factor through $f$. Since $X \in \mathcal{P}$ and $\tau Z \in \mathcal{Q}$, it follows that $X\rightarrow \tau Z$ is in $\radA^{\omega+1}$. Hence, in such a case the sequence $(\ast)$ is not $\mathcal{E}$-exact. We conclude that for $X,Z$ indecomposable
\begin{align*}
    \Ext_\mathcal{E}^1(Z,X) = \begin{cases}
        \Ext_A^1(Z,X) &\textnormal{for }X\notin \mathcal{P}\textnormal{ or }\mathcal{Z}\notin \mathcal{Q},\\
        0              &\textnormal{for }X\in \mathcal{P}\textnormal{ and }\mathcal{Z}\in \mathcal{Q}.
    \end{cases}
\end{align*}

Next, we prove that the sequence
\begin{align*}
    0 \longrightarrow P_1 \longrightarrow R^\lambda_1 \longrightarrow Q_1 \longrightarrow 0
\end{align*}
is almost $\mathcal{E}$-exact for all $\lambda \in k\cup \{\infty \}$. It is not $\mathcal{E}$-exact, since it does not split. We must show that for every morphism $ P_1\rightarrow X$, the morphism $f$ in some pushout diagram
\begin{equation*}
    \begin{tikzcd}
        P_1 \arrow[r] \arrow[d]& R_1 \arrow[d]\\
        X \arrow[r, "f"] & Y
    \end{tikzcd}
\end{equation*}
is an $\mathcal{E}$-monomorphism or the morphism $P_1 \rightarrow X \oplus R_1$ is an {$\mathcal{E}$-monomorphism}. By Remark \ref{indenough} (2) this property only needs to be tested for indecomposable $X$. If $X\not \in \mathcal{P}$, then $X\rightarrow Y$ is an $\mathcal{E}$-monomorphism. If $X\in \mathcal{P}$, then $X\cong P_i$ for some $i \geq 1$. Note that we can assume $P_1 \rightarrow X$ to be a monomorphism, since otherwise it would be zero and $f$ would split. Thus, for $i=1$ the morphism $P_1 \rightarrow X$ is the identity and so $P_1 \rightarrow X\oplus R_1$ splits. For $i > 1$ the morphism $P_1 \rightarrow X$ induces a short exact sequence
\begin{align*}
    0 \longrightarrow P_1 \longrightarrow P_i \longrightarrow Z \longrightarrow 0.
\end{align*}
with $Z \in \mathcal{R}$. Hence, the morphism $P_1 \rightarrow X$ is an $\mathcal{E}$-monomorphism and so is $P_1 \rightarrow X \oplus R_1$. It follows that the sequence in question is almost $\mathcal{E}$-exact.

By Theorem \ref{genericex} we conclude that the exact structure $\mathcal{E}$ corresponds to a generic module. We already knew this, since the closed set corresponding to $\radA^{\omega+1}$ is exactly $\{G\}$. The examination here reproves that $G$ is generic. Moreover, by Proposition \ref{left} an $\bar{\mathcal{E}}$-injective envelope $Q$ of $P_1$ must have $G$ as an indecomposable direct summand.  

\end{document}